 \documentclass[11pt,twoside]{article}
 \usepackage[dvips]{graphics}
 \usepackage{amsmath,graphicx,amsfonts,amssymb}
 \usepackage[mathscr]{euscript}
\usepackage{mwe,tikz}
\usepackage[percent]{overpic}
\usepackage{graphicx}
\usepackage{subfig}

 \newcommand{\nc}[2]{ \newcommand{#1}{#2} }

\nc{\avint}{ {- \hspace{-3.5mm} \int} }  % average integral

\nc{\R}{\rm I \! R}  % Real Numbers
\nc{\N}{\rm I \! N} 
\newcommand{\closure}[1]{ \stackrel{\rule{.1 in}{.01 in}}{#1} }

\newcommand{\chisub}[1]{ {\mathbf{\chi}}_{_{#1}} }

\newcommand{\refeqn}[1]{ (\!\!~\ref{eq:#1}) } % gives references to
\newcommand{\refthm}[1]{ \!\!~\ref{#1} }    % equations or theorems

\nc{\Holder}{H\"{o}lder\ }
\nc{\newnu}{\Upsilon}

\nc{\ith}{ \ensuremath{\text{i}^{\text{th}}} }
\nc{\jth}{ \ensuremath{\text{j}^{\text{th}}} }
\nc{\kth}{ \ensuremath{\text{k}^{\text{th}}} }
\nc{\dst}{ \ensuremath{\text{1}^{\text{st}}_{\delta}} }
\nc{\dnd}{ \ensuremath{\text{2}^{\text{nd}}_{\delta}} }
\nc{\ost}{ \ensuremath{\text{1}^{\text{st}}} }
\nc{\tnd}{ \ensuremath{\text{2}^{\text{nd}}} }
\nc{\curl}{ \nabla \times }
\nc{\Div}{ \nabla \cdot }
\nc{\DC}{K}

\nc{\Ppl}{ \mathcal{M}^{+} }  \nc{\Pmn}{ \mathcal{M}^{-} }

\nc{\smiley}{ $\stackrel{\because}{\smile} \;$ }

\numberwithin{equation}{section}

 \newtheorem{theorem}{Theorem}[section]
 \newtheorem{lemma}[theorem]{Lemma}
 
 \newtheorem{proposition}[theorem]{Proposition}
 \newtheorem{corollary}[theorem]{Corollary}
 \newtheorem{definition}[theorem]{Definition}
 
 \newtheorem{claim}[theorem]{Claim}
 \newtheorem{remark}[theorem]{Remark}
 \def\qed{\hfill\rule{1ex}{1ex}\\}
 \newenvironment{proof}{\noindent {\bf
 Proof.}}{\qed}
 
\title{ \textbf{The geometry of the triple junction between three fluids in equilibrium}}
\author{Blank\footnote{Department of Mathematics, Kansas State University, blanki@math.ksu.edu}, Elcrat\footnote{Department of Mathematics, Wichita State University, deceased}, and Treinen\footnote{Department of Mathematics, Texas State University, rt30@txstate.edu}
\normalsize}

\begin{document}
\maketitle

\begin{abstract}
          We present an approach to the problem of the blow up at the triple junction of three
fluids in equilibrium.    Although many of our results can already be found in the literature, our
approach is almost self-contained and uses the theory of sets of finite perimeter without making use
of more advanced topics within geometric measure theory.  Specifically, using only the calculus of 
variations we prove two monotonicity formulas at the triple junction for the three-fluid configuration,
and show that blow up limits exist and are always cones.  We discuss some of the geometric
consequences of our results.  \\ \ \\
%	We conduct an analysis of the blow up at the triple junction of three fluids in equilibrium.
%Energy minimizers have been shown to exist in the class of functions of bounded variations, and
%the classical theory implies that an interface between two fluids is an analytic surface.  We prove
%two monotonicity formulas at the triple junction for the three-fluid configuration, and show that
%blow up limits exist and are always cones.  We discuss some of the geometric consequences of
%our results.  \\ \ \\
	AMS subject classification: 76B45, 35R35, and 35B65 % Ivan: these are legit!
\end{abstract}

%%%%%%%%%%%%%%%%%%%%%%%%%%%%%%%%%%%%%%%%%%%%%%%%%%

\section{Introduction}    \label{intro}
%
%Here we put the words that encourage our readers to continue.  Like ``feeming to be of ufe'' and ``fo fucceffully''
%and ``Creedmore Dumps.'' Insofar as insofar is insofar the insofar insofar of our insofar is insofar.
Let $\Omega\subset \R^n$ be a bounded domain with boundary smooth enough
that the interior sphere condition holds. Then consider a partition of $\Omega$ into three
sets $E_j$, $j = 0,1,2.$ Each $E_j$ will represent a fluid, and we assume that the three fluids
are immiscible and are in equilibrium with respect to the energy functional
\begin{equation}
\mathcal{F}_{SWP}(\{E_j\}) 
:= \sum_{j = 0}^2 \left( \alpha_j \int_{\Omega} |D\chisub{E_j}|
+ \beta_j \int_{\partial \Omega} \chisub{E_j} \; d\mathcal{H}^{n-1} 
+ \rho_j g \int_{E_j} z \; dV \right)
\label{eq:BigFDefMain}
\end{equation}
where $g$ is determined by the force of gravity, and
where the constants $\alpha_j, \beta_j,$ and $\rho_j$ are determined by constitutive properties of our
fluids.  It will make the most sense to consider sets with finite perimeter, as this functional is infinite otherwise,
and accordingly, we will work within the framework afforded to us by functions of bounded variation.
We will define this functional more carefully and state some assumptions that we will make on the
constitutive constants in Section~\ref{DNB} below.  Two common physical situations where this mathematical model arise
include first, if there is a double sessile drop of two distinct immiscible fluids resting on a surface with air above,
and second,  if a drop of a light fluid is floating on the top of a heavier fluid and below a lighter fluid as would
be the case when oil floats on water and below air.
See Figure~\ref{fig:double_sessile} for an example
of the first situation, and Figure~\ref{fig:ThreeFluids} (found within Section~\ref{DNB}) for an
example of the second situation. The terms in the energy functional given above
arise from (in the order in which they appear) surface tension forces, wetting energy, and the gravitational
potential.

\begin{figure}[!h]
	\centering
	\scalebox{1.2}{\includegraphics{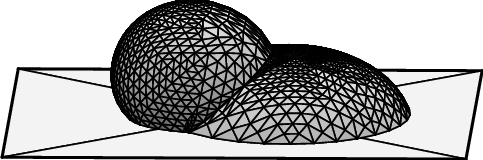}}
	\caption{A double sessile drop.}
	\label{fig:double_sessile}
\end{figure}

In this work we will study the local micro-structure of the triple junction between the fluids.  We prove two
monotonicity formulas, one with a volume constraint, and one without the volume constraint, but which is
sharp in some sense.  Both of these formulas can be compared to the classical Allard monotonicity formula \cite{Al},
and although the formulas we give are obviously not as broad in applicability, they are proven using only tools that
are basic within the calculus of variations and the theory
of sets of finite perimeter.  We use the monotonicity formula to show that blow up limits of the energy minimizing
configurations must be cones, and thus that they are determined completely by their values on the ``blow up sphere.'' 
We then study the implications of minimizing on the blow up sphere for the minimizers in the tangent plane
to the blow up sphere given that the point of tangency is at a triple point.  The consequences are geometric
restrictions on the energy minimizing configurations in the blow up sphere.  Our results can be summarized
in the following theorem:
\begin{theorem}   \label{mainresult}
Assuming that the triple $\{E_j\}$ minimizes the functional\refeqn{BigFDefMain}and assuming that
$x_0 \in \partial E_0 \cap \partial E_1 \cap \partial E_2,$ there exists a blowup limit where the $\partial E_j$
will converge to half-planes containing $x_0,$ and the angles between the half-planes along any blowup limit
satisfy the Neumann Angle Condition:
\begin{equation}\label{eq:enspre}
\frac{\sin\gamma_{01}}{\sigma_{01}} = \frac{\sin\gamma_{02}}{\sigma_{02}} =
\frac{\sin\gamma_{12}}{\sigma_{12}}.
\end{equation}
Here $\gamma_{ij}$ is the angle at the triple point measured within
$E_k$ (where $\{i,j,k\} = \{0,1,2\}$), and $\sigma_{ij}$ is the surface tension at the interface of $E_i$ and
$E_j.$
\end{theorem}
This theorem can also be mostly constructed from the work of Morgan and his students and co-authors
who use advanced topics within the field of geometric measure theory, and we will give a more thorough
comparison in a paragraph below after we first turn to some of the historical background of this problem.

The study of the floating drop problem goes back at least to 1806 when Laplace \cite{La}  formulated the
problem with the assumption of symmetry, and of course, the regularity of the interfaces between the fluids
and also the regularity of the triple junction curve.  In 2004 Elcrat, Neel, and Siegel \cite{ENS}
showed the existence (and, under some assumptions, uniqueness) of solutions for Laplace's formulation, and 
they still assumed the same conditions of symmetry and regularity.  In the time between these results
there were obviously great advancements in the regularity theory involving both the space of functions of
bounded variation and geometric measure theory.  It is with these tools that we will work, and so a quick
survey may be of use for the reader, and so we will provide a very short one in Section~\ref{BBV} below.

The study of soap film clusters began in earnest in the 1970's, and this problem has many connections
with the current work, so a comparison is in order.  In the soap film problem a region of space is partitioned
by sets, and the soap film is modeled by the boundaries of the sets, and the surface areas of these surfaces
are minimized under some volume constraint. The energy is similar to ours, although
it is simpler in some ways.  In particular, there are no weights to the surface tensions (one can set those
to unity), there is no gravitational potential, and there is no wetting term.  The wetting term is the easiest
by far to address, and even the gravitational potential can be dealt with by observing how the surface
tension term becomes much more important in blow up limits, but the fact that in our energy the surface
tension terms vary with each fluid creates considerable new difficulties.
Jean Taylor \cite{Ta} classified the structure of the singularities of soap film clusters, and among other results
was able to show that at triple junction points the surfaces meet at $120^{\circ}$.  Frank Morgan and
collaborators worked on various other aspects of soap bubble clusters, including showing that the standard
double bubble is the unique energy minimizer in a collaboration with Hutchings, Ritor\'e, and Ros \cite{HMRR}.
(See also his book {\em Geometric Measure Theory} \cite{Mo2016} and many references therein.)

It is with this approach that Morgan, White, and others study the problem of three immiscible fluids. Lawlor and
Morgan worked on paired calibrations with immiscible fluids \cite{LM1994}, White used Fleming's flat chains in
order to show the existence of least-energy configurations \cite{W1996}, and then Morgan was able to show
regularity in $\R^2$ and for some cases in $\R^3$ \cite{Mo1998} and he used Allard's monotonicity formula
for varifolds in order to obtain blowup limits.   More recently Morgan returned to the problem in $\R^2$ and
showed under some conditions that a planar minimizer with finite boundary and with prescribed areas consists
of finitely many constant-curvature arcs \cite{Mo2005}.  Although the work just described would yield
most of the conclusions of our main theorem, it is difficult to follow or inaccessible to all but experts
within the field of geometric measure theory.

Our approach is mostly limited to the formulation using functions of bounded variation.  The framework we use is
based on the work of Giusti \cite{G}, where he studies the regularity of minimal surfaces, but it is in a paper by
Massari \cite{Mas} that our problem is first formulated.  Massari showed the existence of energy minimizers, and
commented that Giusti's theory would apply in any region away from a junction of multiple fluids.  Massari and
Tamanini studied a related problem involving optimal segmentations using an approach similar to ours and obtained
a different but analogous monotonicity formula \cite{MT}.  Leonardi \cite{L2001} proved a very useful elimination
theorem about solutions to this problem which roughly states that if the volume of
some fluids is small enough in a ball, then those fluids must not appear in a ball of half the radius.  Two other
references that may be helpful are by Massari and Miranda \cite{MM} and Leonardi \cite{L2000}.  Lastly, Maggi
\cite{Mag} recently published a book that treats some aspects of this problem, including a different proof of
Leonardi's Elimination Theorem.

Finally, we give an outline of our paper.  In Section~\ref{BBV} we collect results on the space of functions
of bounded variation, distilling facts we need from much longer works on the subject.  In Section~\ref{DNB}
we carefully define our problem and some closely related problems, and we discuss some results by Almgren,
Leonardi, and Massari that will be crucial to our work.
%In Section~\ref{infiltration} we refine Leonardi's
%Elimination Theorem for a general concentric ball, and give explicit constants.
In Section~\ref{RR} we show that in the blow up limit it suffices to consider the energy functional that ignores
any wetting energy and any gravitational potential.  In Section~\ref{MSE1} we prove a monotonicity formula
centered about a triple point for the case with volume constraints.  In Section~\ref{MSE2} we drop the volume
constraints and we are able to achieve a sharper monotonicity formula.  At the end of Section~\ref{MSE2} we
give a comparison between our monotonicity formulas and some of the monotonicity formulas that have already
appeared.  In Section~\ref{MC} we use our first monotonicity formula to show that any blow up limit must be a
configuration consisting of cones.  Section~\ref{TP} connects these cones to the blow up sphere.  We then
consider the tangent plane to  a triple point on the blow up sphere, and we are able to show that energy
minimizers in the tangent plane must also be cones.  Finally, in Section~\ref{classification} we show that
those fluids in the tangent plane must be connected and satisfy the same angle condition as was derived in
\cite{ENS}, but we use different methods from them.

%%%%%%%%%%%%%%%%%%%%%%%%%%%%%%%%%%%%%%%%%%%%%%%%%%

\section{Background on Bounded Variation}    \label{BBV}

In the process of studying the two fluid problem, we discovered that some theorems that we needed were either
scattered in different sources, or embedded within the proof of an existing theorem, but not stated explicitly.
For these reasons we have gathered together the theorems that we need here.  Our main sources here were
\cite{AFP}, \cite{EG}, and \cite{G}.  

We assume that $\Omega \subset \R^n$ is an open set with a differentiable boundary. 
We define $BV(\Omega)$ to be the subset of $L^1(\Omega)$ with bounded variation, measured by 
$$\int_\Omega|Df| = \sup \left\{ \int_\Omega f \, \text{div} \phi \,\,:\,\, \phi\in C^1_c(\Omega;\R^n), 
     \mathit{|\phi|\leq 1} \right\},$$
with the corresponding definition of $BV_{loc}(\Omega)$.  We assume some familiarity with these spaces, including,
for example, the basic structure theorem which asserts that the weak derivative of a $BV$ function can be understood
as a vector-valued Radon measure.  (See for example p. 166-167 of \cite{EG}.)  

\begin{theorem}[Density Theorem I]  \label{DT1}
Let $f \in BV(\Omega).$  Then there exists $\{ f_j \} \subset C^{\infty}(\Omega)$ such that
\begin{enumerate}
   \item $$||f_j - f||_{L^1(\Omega)} \rightarrow 0,$$
   \item $$\int_{\Omega} |Df_j| \; dx \rightarrow \int_{\Omega} |Df| \;,$$
   \item $$\text{for any} \ g \in C^{0}_{c}(\closure{\Omega}; \R^n) \ \ \text{we have} \ \ 
		 \mathit{\int_{\Omega} g \cdot Df_j \; dx \rightarrow \int_{\Omega} g \cdot Df \;.}$$
\end{enumerate}
\end{theorem}

\begin{remark}[Not $W^{1,1}$ convergence, but quite close]  \label{nwcbqc}
In any treatment on $BV$ functions care is always taken to emphasize that one does \textit{not}
have $$\int_{\Omega} |D(f_j - f)| \rightarrow 0,$$ in the theorem above.  In particular,
Characteristic functions of smooth sets are in $BV$ but not in $W^{1,1},$ and so $BV$ is
genuinely larger than $W^{1,1}.$  On the other hand, the second part of the theorem above can
be ``localized'' in some useful ways which are not clear from the statement above by itself.
\end{remark}

\begin{theorem}[Density Theorem II]  \label{DT2}
Let $f$ and the $\{ f_j \}$ be taken to satisfy the hypotheses and the conclusions of the theorem above.
Let $\Omega^{\prime} \subset \subset \Omega$ be an open Lipschitz set with 
\begin{equation}
     \int_{\partial \Omega^{\prime}} |Df| = 0 \;.
\label{eq:novonb}
\end{equation}
Then $$\int_{\Omega^{\prime}} |Df_j| \; dx \rightarrow \int_{\Omega^{\prime}} |Df| \;.$$
Furthermore, although simply convolving $f$ (or $f$ extended to be zero outside of $\Omega$) 
with a standard mollifier is insufficient to
produce a sequence of $\{ f_j \}$ with the properties given in the previous theorem, they
will all hold on every $\Omega^{\prime} \subset \subset \Omega$ satisfying Equation \ref{eq:novonb}.
\end{theorem}

\begin{remark}[There are lots of good sets]  \label{MG}
The usefulness of this theorem is unclear until we show the existence of many such $\Omega^{\prime}$
which satisfy Equation \ref{eq:novonb}.  This fact follows from the following theorem found within
Remark 2.13 of \cite{G}.
\end{remark}

\begin{theorem}[Two-sided traces]  \label{tst}
Let $\Omega^{\prime} \subset \subset \Omega$ be an open Lipschitz set and let $f \in BV(\Omega).$  Then 
$f|_{\Omega^{\prime}}$ and $f|_{\Omega^{\prime c}}$ have traces on $\partial \Omega^{\prime}$ which we
call $f_{\Omega^{\prime}}^{-}$ and $f_{\Omega^{\prime}}^{+}$ respectively, and these traces satisfy:
\begin{equation}
     \int_{\partial \Omega^{\prime}} |f_{\Omega^{\prime}}^{+} - f_{\Omega^{\prime}}^{-}| \; d\mathcal{H}^{n-1}
   = \int_{\partial \Omega^{\prime}} |Df|
\label{eq:trform}
\end{equation}
and even $Df = (f_{\Omega^{\prime}}^{+} - f_{\Omega^{\prime}}^{-}) \nu d\mathcal{H}^{n-1}$ where $\nu$ is the
unit outward normal.  Now by taking $\Omega^{\prime} = B_{\rho}(x_0)$ with $x_0 \in \Omega$ then for almost every 
$\rho$ such that $B_{\rho}(x_0) \subset \Omega$ we will have
\begin{equation}
   \int_{\partial B_{\rho}(x_0)} |Df| = 0
\label{eq:goodrho}
\end{equation}
and therefore $f_{\Omega^{\prime}}^{-}(x) = f_{\Omega^{\prime}}^{+}(x) = f(x)$ for $\mathcal{H}^{n-1}$ almost 
every $x \in \partial B_{\rho}(x_0).$
\end{theorem}

From the proof of Lemma~2.4 of \cite{G}, we extract
\begin{theorem}\label{thm:giusti2.4}
Let $\tilde{\mathcal{B}}_R$ denote the ball in $\R^{n-1}$ centered at $0$ with radius $R.$
Let $C^+_R = \tilde{\mathcal{B}}_R \times (0,R)$ and $f\in BV(C^+_R)$.
Let $0<\epsilon^\prime <\epsilon< R$, and set
$Q_{\epsilon,\epsilon^\prime} = \tilde{\mathcal{B}}_R\times(\epsilon^\prime ,\epsilon)$.  Then
\begin{equation}
          \int_{\tilde{\mathcal{B}}_R} |f_\epsilon - f_{\epsilon^\prime}| \, d\mathcal{H}^{n-1}
   \leq \int_{Q_{\epsilon,\epsilon^\prime}} |D_nf|\, dx.
\end{equation}
\end{theorem}
We will need
\begin{lemma}%[5.3 of \cite{G}] 
\label{G5.3}
Let $f \in BV(B_R)$ and $0 < \rho < r < R.$  Then
\begin{equation}
\begin{array}{rl}
   \displaystyle{\int_{\partial B_1} |f^{-}(rx) - f^{-}(\rho x)| \; d\mathcal{H}^{n-1}}
  &\displaystyle{\leq \int_{B_r \setminus B_{\rho}} \left| \left\langle \frac{x}{|x|^n}, Df \right\rangle \right|} \\
\ \\
\ &\text{and} \\
\ \\
   \displaystyle{\int_{\partial B_1} |f^{+}(rx) - f^{+}(\rho x)| \; d\mathcal{H}^{n-1}} 
  &\displaystyle{\leq \int_{\closure{B_r} \setminus \closure{B_{\rho}}} \left| \left\langle \frac{x}{|x|^n}, Df \right\rangle \right| \ .} 
\end{array}
\label{eq:G5.1}
\end{equation}
\end{lemma}
We conclude with Helly's Selection Theorem which is the standard $BV$ compactness theorem:
\begin{theorem}[Helly's Selection Theorem]   \label{Helly}
Given $U \subset \R^n$ and a sequence of functions $\{f_j\}$ in $BV_{loc}(U)$ such that for any
$W \subset \subset U$ there is a constant $C < \infty$ depending only on $W$ which satisfies:
\begin{equation}
   ||f_j||_{BV(W)} := ||f_j||_{L^1(W)} + \int_W |Df_j| \leq C
\label{eq:HellConst}
\end{equation}
then there exists a subsequence $\{ f_{j_k} \}$ and a function $f \in BV_{loc}(U)$ such that on every $W \subset \subset U$
we have
\begin{equation}
   ||f_{j_k} - f||_{L^1(W)} \rightarrow 0
\label{eq:HellConv1}
\end{equation}
and
\begin{equation}
   \int_W |Df| \leq \liminf \int_W |Df_{j_k}| \;.
\label{eq:HellConv2}
\end{equation}
\end{theorem}

%%%%%%%%%%%%%%%%%%%%%%%%%%%%%%%%%%%%%%%%%%%%%%%%%%%%%%
%%%%%%%%%%%%%%%%%%%%%%%%%%%%%%%%%%%%%%%%%%%%%%%%%%%%%%

\section{Definitions, Notation, and more Background}  \label{DNB}

We denote the surface tension at the interface of $E_i$ and $E_j$ with $\sigma_{ij},$ we
use $\beta_i$ as the coefficient that determines the wetting energy of $E_i$ on the
boundary of the container, we let $\rho_i$ be the density of the \!\ith fluid, and we use
$g$ as the gravitational constant.  The domain $\Omega$ is the container, and we assume
$B_1 \subset \subset \Omega \subset \R^n.$  We define

% Fix $B_1 \subset \Omega \subset \R^3.$  Let $\vec{\alpha}, \vec{\beta}, \vec{\rho} \in \R^3.$

\begin{equation}
  \begin{array}{rl}
     \alpha_0 \!\!&:= \frac{1}{2}(\sigma_{01} + \sigma_{02} - \sigma_{12}) \\
\ \\
     \alpha_1 \!\!&:= \frac{1}{2}(\sigma_{01} + \sigma_{12} - \sigma_{02}) \\
\ \\
     \alpha_2 \!\!&:= \frac{1}{2}(\sigma_{02} + \sigma_{12} - \sigma_{01}) \;, \\
  \end{array}
\label{eq:sigs}
\end{equation}
and we will assume
\begin{equation}
   \alpha_j > 0, \ \ \text{for all} \ j
\label{eq:strtri}
\end{equation}
throughout our paper and refer to this condition as the strict triangle inequality, but note that this
condition is frequently called the strict triangularity hypothesis.  (See \cite{L2001} for example.)

\begin{definition}[Permissible configurations]  \label{permconf}
The triple of open sets $\{E_j\}$ is said to be a permissible configuration or more simply ``permissible'' if
\begin{enumerate}
   \item The $E_j$ are sets of finite perimeter.
   \item The $E_j$ are disjoint.
   \item The union of their closures is $\closure{\Omega}.$
% \item The volumes are prescribed: $|E_j| = v_j$ for $j=0,1,2$.
\end{enumerate}
In a case where volumes are prescribed, in order for sets to be \textit{V-permissible}
we will add to this list a fourth item:
\begin{itemize}
    \item[4.] The volumes are prescribed: $|E_j| = v_j$ for $j=0,1,2$.
\end{itemize}
\end{definition}
See Figure~\ref{fig:ThreeFluids}.

\begin{figure}[!h]
	\centering
	\scalebox{.45}{\includegraphics{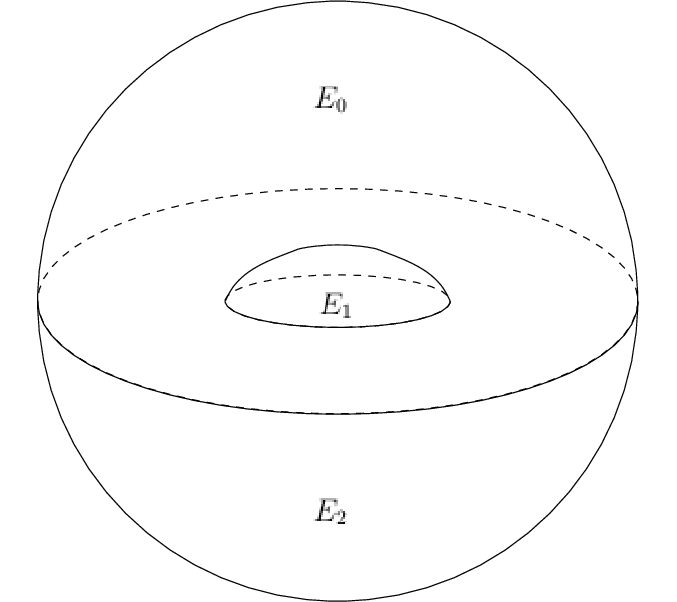}}
	\caption{Permissible sets $\{E_j\}$.}
	\label{fig:ThreeFluids}
\end{figure}

\noindent
The full energy functional which sums surface tension, wetting energy, and potential energy due to gravity is given by:
\begin{equation}
  \mathcal{F}_{SWP}(\{E_j\}) 
    := \sum_{j = 0}^2 \left( \alpha_j \int_{\Omega} |D\chisub{E_j}|
                           + \beta_j \int_{\partial \Omega} \chisub{E_j} \; d\mathcal{H}^{n-1} 
                           + \rho_j g \int_{E_j} z \; dV \right),
  \label{eq:BigFDef}
\end{equation}

\noindent
As we scale inward we can eliminate the wetting energy entirely and view our solution restricted to an
interior ball as a minimizer of an energy given by:
\begin{equation}
  \mathcal{F}_{SP}(\{E_j\}) 
    := \sum_{j = 0}^2 \left( \alpha_j \int_{\Omega} |D\chisub{E_j}|
                           + \rho_j g \int_{E_j} z \; dV \right).
  \label{eq:SPDef}
\end{equation}
Of course this energy we will frequently consider on subdomains, so for
$\Omega^{\prime} \subset \subset \Omega$ we define:
\begin{equation}
  \mathcal{F}_{SP}(\{E_j\},\Omega^{\prime}) 
    := \sum_{j = 0}^2 \left( \alpha_j \int_{\Omega^{\prime}} |D\chisub{E_j}|
                           + \rho_j g \int_{E_j \cap \Omega^{\prime}} z \; dV \right).
  \label{eq:SPODef}
\end{equation}
Massari showed that this energy functional is
lower semicontinuous in \cite{Mas} under certain assumptions on the constants.  (In fact he showed it for
$\mathcal{F}_{SWP},$ but where $\beta_j \equiv 0$
is allowed.)  The lower semicontinuity of $\mathcal{F}_{SP}$ ensures that this Dirichlet problem is
well-posed, although it does not guarantee that the Dirichlet data is attained in the usual sense.  In fact, a minimizer
can actually have any Dirichlet data, but if it does not match up with the given data, then it must pay for an interface
at the boundary.  Summarizing these statements from \cite{Mas} we can say:

\begin{theorem}[Massari's Existence Theorem]   \label{BackMass}
If 
\begin{equation} 
    \begin{array}{rl}
                 \alpha_j &\!\!\!\geq 0, \\
                 \alpha_i + \alpha_j &\!\!\!\geq |\beta_i - \beta_j|,
    \end{array}
\label{eq:nowig}
\end{equation}
for $i,j = 0,1,2,$ if $v_0 + v_1 + v_2 = |\Omega|,$ and
if $\Omega$ satisfies an interior sphere condition, then there exists a minimizer to
$\mathcal{F}_{SWP}$ among permissible triples $\{ E_j \}$ with $|E_j| = v_j.$
The same statement is true if $\mathcal{F}_{SWP}$ is replaced by
either $\mathcal{F}_{SP}$ or $\mathcal{F}_{S}.$  ($\mathcal{F}_{S}$ is defined below.)
Assuming that we allow a two sided trace of our BV characteristic functions on the boundary of our domain,
and making the same assumptions as above, then there will also exist minimizers which satisfy given Dirichlet
data.  (Of course one should refer to the discussion above regarding the nature of Dirichlet data for this
problem.)
\end{theorem}

\begin{remark}[Appropriate Problems]  \label{AppProbs}
It seems worthwhile to observe here the necessity of prescribing Dirichlet data in any problem without a volume
constraint.  Indeed, without a volume constraint or Dirichlet data, one expects two of the three fluids to
vanish in any minimizer.  On the other hand, once you have a volume constraint, you can study the minimizers
both with and without Dirichlet data.
\end{remark}

At this point, we standardize our language for the type of minimizer that we are considering in order to 
prevent language from becoming too cumbersome.
\begin{definition}[Types of Minimizers]  \label{MinTypes}
We will use the syntax:
$$[\; \text{Qualifier(s)}\;]\text{-minimizer} \ \text{of} \ [\; \text{functional} \;] \ \text{in} \ [\; \text{set} \;] \;.$$
The ``qualifiers'' we will use are ``D'' and/or ``V'' to indicate a Dirichlet or a volume constraint
respectively.  So a typical appearance might look like:  $\{E_j\}$ is a
V-minimizer of $\mathcal{F}_{SP}$ in $B_3,$ which means that $\{E_j\}$ are V-permissible
and minimize $\mathcal{F}_{SP}$ in $B_3$ among all V-permissible sets.  If the set is not specified,
then we will assume that the minimization happens on $\Omega$;  If the functional is not specified,
then we assume that $\mathcal{F}_{S}$ is the functional being minimized.  The set given will typically be
bounded, but when it is not bounded we will assume that anything which we call any kind of minimizer will
minimize the given functional when restricted to any compact subset of the unbounded domain.
\end{definition}

\begin{remark}[On restrictions and rescalings]  \label{ResRes}
It is also worth remarking that after restricting and rescaling, a triple which used to
V-minimize some functional will still V-minimize some functional in the new set, but
except in the case of the three cones, the new sets will typically be competing against
V-permissible triples with different restrictions on the volume of each set from the
restrictions at the outset.
\end{remark}

\begin{remark}[Reversal of inclusions]   \label{RevInc}
We also observe that the inclusions of types of minimizers are also reversed from what one might
assume before thinking about it.  In typical set inclusions of this sort, one assumes that more
constraints lead to a smaller set.  Here, because it is the competitors which are being constrained,
the inclusions work in reverse.  Indeed, the set of all DV-minimizers contains
both the set of V-minimizers and the set of D-minimizers insofar as if you take the DV-minimizer
where you take the Dirichlet data to be rather ``wiggly'' then you only compete against other
configurations with similarly wiggly boundary data.  Thus, you are automatically the DV-minimizer
by construction, but you are not likely to be a V-minimizer, as any V-minimizer would prefer less wiggly
boundary data.
\end{remark}

\noindent
Since we intend to study the local microstructure at triple points which are in the interior of $\Omega,$
it will be useful to study the simplified energy functional which ignores the wetting energy and the potential
energy.  By scaling in toward a triple point, we can be sure that the forces of surface tension are much stronger than
the gravitational forces in our local picture, and at the same time the wetting energy will become totally irrelevant, 
as the boundary of $\Omega$ can be scaled away altogether if we zoom in far enough.  So, with these ideas in mind 
we define the simplified energy functional by:
\begin{equation}
  \mathcal{F}_S(\{E_j\})
    := \sum_{j = 0}^2 \left( \alpha_j \int_{\Omega} |D\chisub{E_j}| \right) , \ \ \text{and} \ \
  \mathcal{F}_S(\{h_j\})
    := \sum_{j = 0}^2 \left( \alpha_j \int_{\Omega} |Dh_j| \right) \;.
  \label{eq:Energy}
\end{equation}

\noindent
The energy on $\Omega^{\prime} \subset \subset \Omega$ is:
\begin{equation}
  \mathcal{F}_S(\{E_j\}, \Omega^{\prime})
    := \sum_{j = 0}^2 \left( \alpha_j \int_{\Omega^{\prime}} |D\chisub{E_j}| \right) \;, \ \ \text{and} \ \
  \mathcal{F}_S(\{h_j\}, \Omega^{\prime})
    := \sum_{j = 0}^2 \left( \alpha_j \int_{\Omega^{\prime}} |Dh_j| \right) \;.
  \label{eq:opEnergy}
\end{equation}

\noindent
Let $\Omega^{\prime} \subset \subset \Omega,$ let $\{E_j\}$ be permissible.  Using ``spt'' for ``support'', we define
\begin{equation}
   \newnu(\{E_j\}, \Omega^{\prime}) := \inf \{ \mathcal{F}_{S}(\{\tilde E_j\}) \; : \; 
       \text{spt}(\chisub{E_j} - \chisub{\tilde E_j}) \subset \Omega^{\prime}  \text{ and } \ \{\tilde E_j\}  \text{ is perm.} \}
\label{eq:nudef}
\end{equation}
\begin{equation}
   \Psi(\{E_j\}, \Omega^{\prime}) := \mathcal{F}_{S}(\{E_j\}, \Omega^{\prime}) - \newnu(\{E_j\}, \Omega^{\prime}) \;.
\label{eq:psidef}
\end{equation}
Now assume further that $\{E_j\}$ is V-permissible.  Then we define
\begin{equation}
   \newnu_V(\{E_j\}, \Omega^{\prime}) := \inf \{ \mathcal{F}_{S}(\{\tilde E_j\}) \; : \; 
       \text{spt}(\chisub{E_j} - \chisub{\tilde E_j}) \subset \Omega^{\prime}  \text{ and } \ \{\tilde E_j\}  \text{ is V-perm.} \}
\label{eq:nuVdef}
\end{equation}
\begin{equation}
   \Psi_V(\{E_j\}, \Omega^{\prime}) := \mathcal{F}_{S}(\{E_j\}, \Omega^{\prime}) - \newnu_V(\{E_j\}, \Omega^{\prime}) \;.
\label{eq:psiVdef}
\end{equation}
So $\newnu$ and $\newnu_V$ give the value of the minimal energy configuration with the same boundary
data, while $\Psi$ and $\Psi_V$ give
the amount that $\{E_j\}$ deviates from minimal.  Notice that we are minimizing over the class of sets of finite
perimeter, not over all of BV.

Of course the existence theorem does not address any of the regularity questions near a triple point
and the regularity questions near the boundary of only two of the fluids is already well-understood.
On the other hand, in order to understand the microstructure of triple points which are not located on the boundary of 
$\Omega$ it should suffice to study minimizers of the simplified energy functional, $\mathcal{F}_{S},$ as we have
described above.  We make this heuristic argument rigorous in Section~\ref{RR}, but we still need two more
tools from the background literature.

The first tool we need is a very nice observation due to F.~Almgren which allowed him to virtually ignore volume constraints
when studying the regularity of minimizers of surface area under these restrictions.  Since our energy is bounded
from above and below by a constant times surface area, we can adapt his result to our situation immediately.
\begin{lemma}[Almgren's Volume Adjustment Lemma]   \label{PfVA}
Given any permissible triple $\{ E_j \},$ there exists a $C > 0,$ such that very small volume adjustments can be
made at a cost to the energy which is not more than $C$ times the volume adjustment.  Stated quantitatively:
\begin{equation}
     \Delta \mathcal{F}_{S} \leq C \sum_{j = 0}^2 |\Delta V_j|
\label{eq:val}
\end{equation}
where $\Delta V_j$ is the volume change of $E_j.$
\end{lemma}
This result can be found in \cite{Am1976} (see V1.2(3)) or in \cite{Mo1994} as Lemma 2.2.

The next tool we need is an ``elimination theorem'' which in our setting is due to Leonardi.  (See Theorem 3.1 of \cite{L2001}.)

\begin{theorem}[Leonardi's Elimination Theorem]  \label{LET}
Under the assumptions above, including the strict triangle inequality (Equation\refeqn{strtri}\!\!), if
$\{ E_j \}$ is a V-minimizer, then $\{ E_j \}$ has the elimination property.  Namely, there exists a constant
$\eta > 0,$ and a radius $r_0$ such that if $0 < \rho < r_0,$ $B_{r_0} \subset \Omega,$ and
\begin{equation}
|E_i \cap B_\rho(x)| \leq \eta \rho^n \, ,
\end{equation}
then
\begin{equation}
|E_i \cap B_{\rho/2}(x)| = 0 \;.
\end{equation}
\end{theorem}

%%%%%%%%%%%%%%%%%%%%%%%%%%%%%%%%%%%%%%%%%%%%%%%%

\section{Restrictions and Rescalings}    \label{RR}

We start with a rather trivial observation:  If $\{ E_j \}$ is a V-minimizer of $\mathcal{F}_{SWP}$ among
V-permissible triples, and $B_r \subset \subset \Omega,$ then the triple:  $\{ E_j \cap B_r\}$
DV-minimizes $\mathcal{F}_{SP}$ in $B_r$ among V-permissible triples with Dirichlet data given by the traces of
the $\{ E_j \}$ on the outer boundary of $B_r,$ and whose volumes are prescribed to be the volume of each
$E_j$ intersected with $B_r.$  If this statement were false, then we would immediately get an improvement to
our V-minimizer of $\mathcal{F}_{SWP},$ by replacing things within $B_r.$

Recalling that $B_1 \subset \subset \Omega,$ we wish to define rescalings of our triples and study their properties
in the hopes of producing blowup limits.  For $\lambda \in \R^{+}$ we define $\lambda E_j$ to be the dilation of
$E_j$ by $\lambda.$  In particular,
$$x \in \lambda E_j \ \Longleftrightarrow \ \frac{x}{\lambda} \in E_j \,.$$
Now assume that $\{ E_j \}$ is a D-minimizer of $\mathcal{F}_{SWP}$ in $\Omega,$ and fix $0 < \lambda < 1.$
By virtue of the fact that $\{ E_j \}$ is a D-minimizer of $\mathcal{F}_{SP}$ in $B_\lambda,$ we can scale
our triple $\{ E_j \}$ to the triple $\{ \lambda^{-1} E_j \},$ and easily verify that the new triple is a D-minimizer of
the functional:

\begin{equation}
  \mathcal{F}_{SP\lambda}(\{A_j\},B_1) 
    := \sum_{j = 0}^2 \left( \alpha_j \int_{B_1} |D\chisub{A_j}|
                           + \lambda \rho_j g \int_{A_j \cap B_1} z \; dV \right).
  \label{eq:rescDEF}
\end{equation}
From here, after observing that it is immediate that the characteristic functions corresponding to the triple
$\{ \lambda^{-1} E_j \}$ will be uniformly bounded in $BV(B_1),$ we can apply Helly's selection theorem
(given above as Theorem\refthm{Helly}\!\!) to guarantee
the existence of a blow up limit in $BV.$  More importantly, the blowup limit will be a minimizer of
$\mathcal{F}_{S}.$  For convenience, define $\chisub{E_{j,\lambda_i}} := \chisub{\lambda^{-1}_iE_j}$.

\begin{theorem}[Existence of blowup limits]  \label{EBLOW}
Assume that $\{ E_j \}$ is a D-minimizer or a V-minimizer of $\mathcal{F}_{SP}$ in $\Omega.$
%triple which minimizes $\mathcal{F}_{SWP}$ in $\Omega.$
% $E_{i,\lambda}(x) := \chisub{\lambda E_i}(x),$ then t
In either case, there exists a configuration (which we
will denote by $\{ E_{j, 0} \}$) and a sequence of $\lambda_i \downarrow 0$ such that for each $j:$
\begin{equation}
    ||\chisub{E_{j,\lambda_i}} - \chisub{E_{j, 0}}||_{L^1(B_1)} \rightarrow 0 \ \ \ \text{and} \ \ \
      D\chisub{E_{j,\lambda_i}} \stackrel{\ast}{\rightharpoonup}   D\chisub{E_{j, 0}} \;.
\label{eq:BVcomp}
\end{equation}
Furthermore, the triple $\{ E_{j, 0} \}$ is a D-minimizer of $\mathcal{F}_{S}$ for whatever Dirichlet data it
has in the first case or a V-minimizer of $\mathcal{F}_{S}$ for whatever volume constraints it satisfies in the
second case.
\end{theorem}
\begin{proof}
Based on the discussion preceding the statement of the theorem, it remains to show that $\{ E_{j, 0} \}$ is a
minimizer of $\mathcal{F}_{S}$ under the appropriate constraints.  Lower semicontinuity of the BV norm implies that
$$\mathcal{F}_{S}(\{ E_{j, 0} \}) \leq \liminf_{j \rightarrow \infty} \mathcal{F}_{S}(\{ E_{j, \lambda_i} \}).$$
While on the other hand
\begin{equation}
     \mathcal{F}_{SP\lambda_i}(\{ E_{j, \lambda_i} \})
             = \min \{ \mathcal{F}_{SP\lambda_i}(\{ A_j \}) \, : \, \{ A_j \} \, \text{ is permissible} \}
                                                                                       \leq \mathcal{F}_{SP\lambda_i}(\{ E_{j, 0} \})
\label{eq:minalongway}
\end{equation}
since $\{ E_{j, \lambda_i} \}$ is a minimizer.  Because the gravitational term is going to zero, it is clear that
\begin{equation}
    \mathcal{F}_{S}(\{ E_{j, 0} \}) = \lim_{i \rightarrow \infty} \mathcal{F}_{S}(\{ E_{j, \lambda_i} \}) \;,
\label{eq:limitswork}
\end{equation}
and for the same reason, for any $\epsilon > 0,$ if $\lambda$ is sufficiently small and $i$ is sufficiently large,
then we must have:
\begin{equation}
      |\mathcal{F}_{SP\lambda}(\{ E_{j, 0} \}) - \mathcal{F}_{SP\lambda}(\{ E_{j, \lambda_i} \})| < \epsilon \,.
\label{eq:gravsmall}
\end{equation}

Now if $\{ E_{j, 0} \}$ is not a D-minimizer or V-minimizer (according to the case we are in),
then there exists a D or V-minimizing triple
$\{ \tilde{E}_{j, 0} \}$ and a $\gamma > 0,$ such that
$\mathcal{F}_{S}(\{ E_{j, 0} \}) - \gamma = \mathcal{F}_{S}(\{ \tilde{E}_{j, 0} \}).$
In this case, for all sufficiently small $\lambda,$ we will automatically have
\begin{equation}
\mathcal{F}_{SP\lambda}(\{ E_{j, 0} \}) - \gamma/2 \geq \mathcal{F}_{SP\lambda}(\{ \tilde{E}_{j, 0} \}),
\end{equation}
but then by using Equations\refeqn{minalongway}and\refeqn{gravsmall}we will get a contradiction by observing
that for small enough $\lambda_i$ we will have:
\begin{equation}
\mathcal{F}_{SP\lambda_i}(\{ \tilde{E}_{j, 0} \}) < \mathcal{F}_{SP\lambda_i}(\{ E_{j, \lambda_i} \}) \,.
\end{equation}
\end{proof}

%%%%%%%%%%%%%%%%%%%%%%%%%%%%%%%%%%%%%%%%%%%%%%%%

\section{The Monotonicity of Scaled Energy (Part I)}
\label{MSE1}

\begin{theorem}
Suppose $\{E_j\}\in BV(B_R)$ is V-permissible and $0<\rho<r<R$ with $0\in\cap_{j=0}^2 \partial E_j$.
Then there exists a constant $C$ such that
\begin{eqnarray}
&& \sum^2_{j=0}\alpha_j\left\{ \int_{B_r\setminus B_\rho}
    \left| \left\langle \frac{x}{|x|^n},D\chi_{E_j}\right\rangle \right|\, dx \right\}^2 \nonumber\\
&\leq&
2\sum^2_{j=0}\int_{B_r\setminus B_\rho} |x|^{1-n}|D\chi_{E_j}|\, dx \cdot \nonumber\\
&& \quad
\cdot \Bigg\{ r^{1-n}\mathcal{F}_S\left(\{E_j\},B_r\right)
        - \rho^{1-n}\mathcal{F}_S\left(\{E_j\},B_\rho\right)
       + (n-1)\int^r_\rho t^{-n}\Psi_V \left(\{E_j\},B_t\right)\, dt \nonumber\\
&&  \left.\qquad -\sum_{j = 0}^2\frac{\alpha_j}{8}\int_{B_r\setminus B_\rho} |x|^{1-n}  \left\langle \frac{x}{|x|}, \frac{D\chi_{E_j}}{|D\chi_{E_j}|} \right\rangle^4 |D\chi_{E_j}|\, dx + C(r - \rho) \right\} .
\label{eqn:full_mono}
\end{eqnarray}
\end{theorem}
This estimate and the argument below should be compared with \cite[Lemma 5]{MT} and \cite[Chapter 5]{G}. \\
\begin{proof}
Let $t\in(0,R)$ be such that $0<\rho\leq t\leq r <R$. 
Theorem 2, p. \!\!\!\!\! 172 of \cite{EG} (or similar) implies there exist smooth
functions $f_j(x;\epsilon)$ so that if $\epsilon\rightarrow 0$, then 
$f_j(x;\epsilon) \rightarrow \chi_{E_j}(x)$ in $L^1(B_R)$ and
$$
\int_{B_R} |D\chi_{E_j}| = \lim_{\epsilon\rightarrow 0} \int_{B_R} |Df_j(x;\epsilon)|\, dx.
$$
Then define the conical projection on these smooth functions:
\begin{equation}\label{eqn:conical}
f_{j,t} = f_j(x;\epsilon,t) = \left\{  \begin{array}{ll}
           f_j(x;\epsilon) \ \ \ \ \ &|x| \geq t \\
           f_j\left(\frac{tx}{|x|};\epsilon\right) \ \ \ \ \ &|x| < t \;.
       \end{array}  \right.
\end{equation}
%Cheeseburger
An example of this process can be seen in Figure~\ref{fig:fcts}.
\begin{figure}[htb]
\centering
\subfloat[Level curves for $f_j$]{
    \includegraphics[width=.48\textwidth]{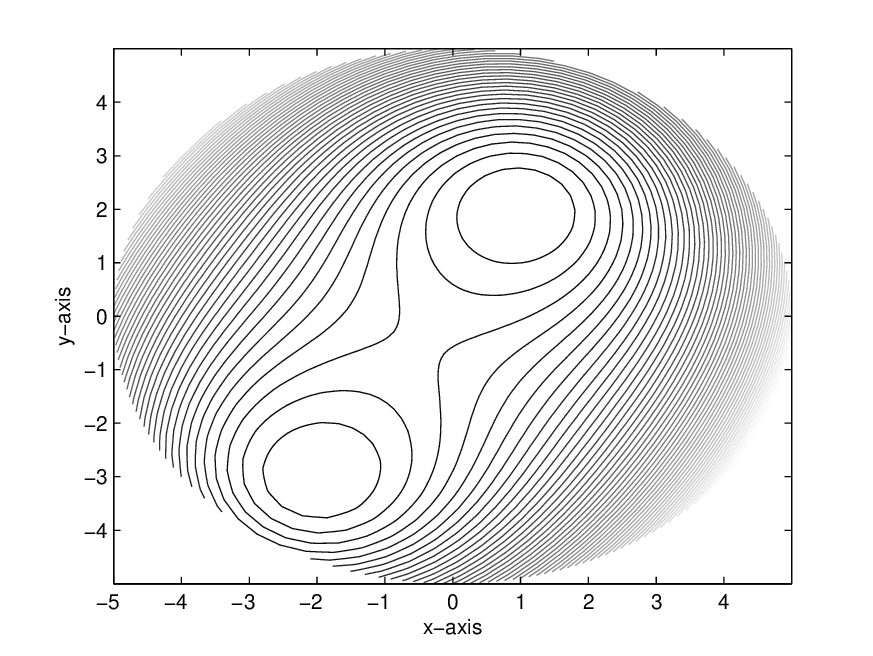}}
\subfloat[Level curves for $f_{j,3}$]{
    \includegraphics[width=.48\textwidth]{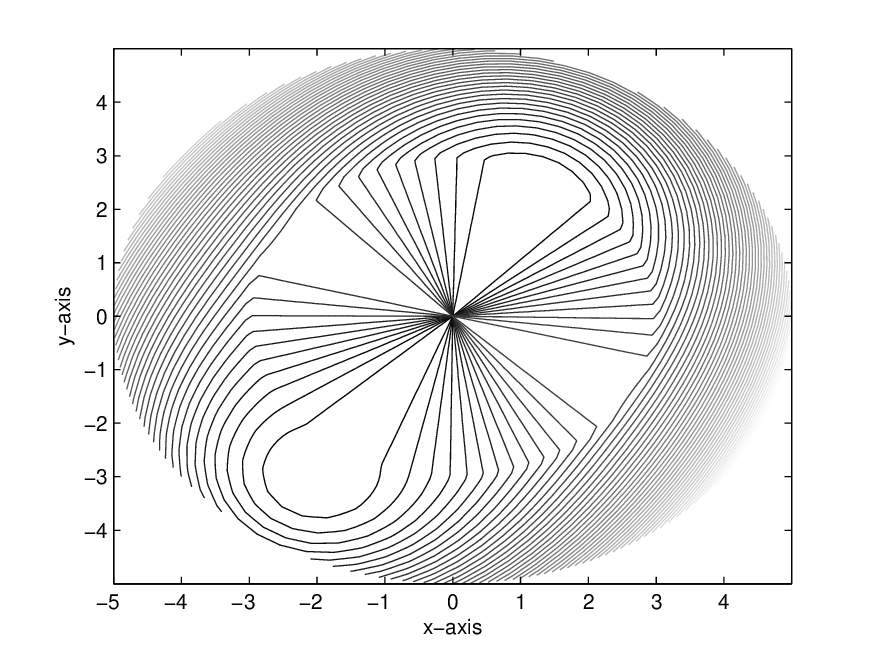}}
\caption{An example: $f_j(x,y) := [(x-1)^2+(y-2)^2] \cdot [(x+2)^2+(y+3)^2]$} \label{fig:fcts}
\end{figure}
With these conical functions we have
\begin{equation}\label{eqn:sqrt}
      \int_{B_t}|Df_{j,t}|\, dx = 
              \frac{t}{n-1}    \int_{\partial B_t} |D f_j|  \left\{  1-\frac{\langle x, D f_j \rangle^2}{|x|^2|D f_j|^2}  \right\}^{1/2}
                                              \, d\mathcal{H}^{n-1}  \mbox{ a.e in $t$}.
\end{equation}
Then $\{E_j\}$ V-permissible implies if $\epsilon \rightarrow 0$, then
$f_{j,t}(x;\epsilon,t)\rightarrow \chi_{\tilde E_j}$ for some {\em set}
$\tilde E_j$ for $j=0,1,2$.  It follows from the V-permissibility of $\{E_j\}$
that $\{\tilde E_j\}$ have the properties that $\tilde E_j \cap \tilde E_i = \emptyset$
for $i\ne j$ and that $\cup \textrm{ closure}\,(\tilde E_j) = B_R$.  It remains to show
that each $\tilde E_j$ is a set of finite perimeter.  Notice that 
\begin{eqnarray}
\int_{B_t}|Df_{j,t}|\, dx &=& \frac{t}{n-1}\int_{\partial B_t}
       |D f_j|\left\{1-\frac{\langle x, D f_j \rangle^2}{|x|^2|D f_j|^2}\right\}^{1/2}\, d\mathcal{H}^{n-1} \nonumber\\
 &\leq& \frac{t}{n-1} \int_{\partial B_t} |D f_j|\, d\mathcal{H}^{n-1} \nonumber\\
&<& \infty \quad \textrm{a.e. in } t,
\end{eqnarray}
then Theorem\refthm{Helly}states that there is a subsequence
$f_j(x;\epsilon_k,t)$ converging in $L^1$ to $\tilde f_j(x;t)\in BV(B_R)$
where the total variations converge as well.  Thus $\tilde E_j$ are sets of
finite perimeter, and $\{\tilde E_j\}$ is permissible, but the volume constraints which will be off by an amount 
controlled by $Ct^n.$  Thus, by applying
Almgren's Volume Adjustment Lemma (see Lemma \ref{PfVA}), we get:
$$
\newnu_V \left(\{E_j\},B_t\right) \leq \mathcal{F}_S\left(\{\tilde E_j\},B_t\right) + C t^n
       = \lim_{\epsilon\rightarrow 0} \mathcal{F}_S\left(\{f_j(x;\epsilon,t)\},B_t\right) + Ct^n.
$$
Then by using
$$
\newnu_V \left(\{E_j\},B_t\right) = \mathcal{F}_S\left(\{E_j\},B_t\right) - \Psi_V \left(\{E_j\},B_t\right),
$$
with Equation \eqref{eqn:sqrt} and the Taylor series for $\sqrt{1-x}$ at $0$ with $x>0$ small, we obtain
\begin{eqnarray}
& & \mathcal{F}_S\left(\{E_j\},B_t\right) - \Psi_V \left(\{E_j\},B_t\right) \nonumber\\
&\leq& \mathcal{F}_S\left(\{ \tilde{E}_j\},B_t\right) + Ct^n \nonumber\\
&\leq& \lim_{\epsilon\rightarrow 0}\sum^2_{j=0} \frac{t\alpha_j}{n-1}
      \left(\int_{\partial B_t} |D  f_j|\, d\mathcal{H}^{n-1}
           - \frac{1}{2}\int_{\partial B_t} \frac{\left\langle x,D f_j\right\rangle^2}{|x|^2|D f_j|}\,
                           d\mathcal{H}^{n-1} \right. \nonumber\\
& &     \left. \hspace{1.1in} - \frac{1}{8}\int_{\partial B_t} \frac{\left\langle x,D f_j\right\rangle^4}{|x|^4|D f_j|^3}\, 
                           d\mathcal{H}^{n-1}
      \right) + Ct^n. \qquad
\end{eqnarray}
Then by rearranging terms and multiplying through by $(n-1)t^{-n}$ we get:
\begin{eqnarray}
&&\lim_{\epsilon\rightarrow 0}\sum^2_{j=0} \alpha_j \frac{t^{1-n}}{2}
               \int_{\partial B_t} \frac{\left\langle x, D f_j \right\rangle^2}{|x|^2|Df_j|}\, d\mathcal{H}^{n-1} \nonumber\\
&\leq& -(n-1)t^{-n} \mathcal{F}_S \left(\{E_j\},B_t\right) + (n-1) t^{-n} \Psi_V \left(\{E_j\},B_t\right)  \nonumber \\
&& \quad + \; \lim_{\epsilon\rightarrow 0}t^{1-n}\sum^2_{j=0} \alpha_j \int_{\partial B_t} |D f_j| \, d\mathcal{H}^{n-1}
            - \lim_{\epsilon\rightarrow 0} \sum^2_{j=0} \frac{\alpha_j}{8}t^{1-n}
              \int_{\partial B_t} \frac{\left\langle x,D f_j\right\rangle^4}{|x|^4|D f_j|^3}\,
                                    d\mathcal{H}^{n-1}  + C \quad\qquad \\
&=& \lim_{\epsilon\rightarrow 0}\left[ -(n-1)t^{-n} \mathcal{F}_S\left(\{f_j\},B_t\right)
              + t^{1-n}\sum^2_{j=0} \alpha_j \int_{\partial B_t} |D f_j| \, d\mathcal{H}^{n-1}  \right]   \nonumber\\
&& \quad + \; (n-1) t^{-n} \Psi_V \left(\{E_j\},B_t\right) 
              - \lim_{\epsilon\rightarrow 0} \sum^2_{j=0} \frac{\alpha_j}{8}
                    \int_{\partial B_t} |x|^{1-n}
        \left\langle \frac{x}{|x|},\frac{D f_j}{|D f_j|}\right\rangle^4 |D f_j|\, d\mathcal{H}^{n-1} + C \quad\qquad \\
&=& \lim_{\epsilon\rightarrow 0}\left\{\frac{d\,}{dt}
             \left[ t^{1-n} \mathcal{F}_S\left(\{f_j\},B_t\right)\right]
                 - \sum^2_{j=0} \frac{\alpha_j}{8}\int_{\partial B_t} |x|^{1-n}
                 \left\langle \frac{x}{|x|},\frac{D f_j}{|D f_j|}\right\rangle^4 |D f_j|\, d\mathcal{H}^{n-1}\right\}  \nonumber\\
&& \quad + \; (n-1) t^{-n} \Psi_V \left(\{E_j\},B_t\right)\, +C \ \ \text{a.e.}\, t.
\end{eqnarray}
Integrating with respect to $t$ between $\rho$ and $r$, we have
\begin{eqnarray}
&&\lim_{\epsilon\rightarrow 0}\sum^2_{j=0} \frac{\alpha_j}{2}
             \int_{B_r\setminus B_\rho} \frac{\left\langle x, Df_j\right\rangle^2}{|x|^{n+1}|D f_j|}\, dx \nonumber\\
&\leq& r^{1-n} \mathcal{F}_S\left(\{E_j\},B_r\right)
             - \rho^{1-n} \mathcal{F}_S\left(\{E_j\},B_\rho\right) 
            + \; (n-1)\int^r_\rho t^{-n} \Psi_V \left(\{E_j\},B_t\right)\, dt  \label{eqn:integrating}\\
&& \quad - \lim_{\epsilon\rightarrow 0} \frac{\alpha_j}{8}
                                \int_{B_r\setminus B_\rho} |x|^{1-n}\left\langle \frac{x}{|x|},\frac{D f_j}{|D f_j|}\right\rangle^4 
                                        |D f_j|\, dx \nonumber
                + C(r - \rho) \,.
\end{eqnarray}
Finally, the Schwartz inequality implies
\begin{eqnarray}
&&\sum^2_{j=0}\left\{\lim_{\epsilon\rightarrow 0}\alpha_j\int_{B_r\setminus B_\rho}
               \left| \left\langle \frac{x}{|x|^n}, Df_j\right\rangle\right|\, dx \right\}^2 \nonumber\\
&\leq& \lim_{\epsilon\rightarrow 0}\left\{\sum^2_{j=0}
              \alpha_j\int_{B_r\setminus B_\rho} |x|^{1-n}|Df_j|\, dx \int_{B_r\setminus B_\rho}
                      \frac{\left\langle x,Df_j \right\rangle^2 }{|x|^{n+1}|Df_j|}\, dx\right\} \\
&\leq& 2\left(\lim_{\epsilon\rightarrow 0}\sum^2_{j=0}
              \int_{B_r\setminus B_\rho} |x|^{1-n}|Df_j|\, dx \right)\left(\lim_{\epsilon\rightarrow0} 
                   \sum^2_{j=0}\frac{\alpha_j}{2}\int_{B_r\setminus B_\rho}
                  \frac{\left\langle x,Df_j \right\rangle^2 }{|x|^{n+1}|Df_j|}\, dx\right). \quad\qquad
\end{eqnarray}
The result follows by combining the preceding with \eqref{eqn:integrating} and the application of Theorem~3, p.~175 in \cite{EG}.

\end{proof}

\begin{corollary}
Suppose $\{E_j\}\in BV(B_R)$ is V-permissible and is made up of sets of finite perimeter and $0<\rho<r<R$.  Further, suppose $\Psi_V \left(\{E_j\}\right) \equiv 0$. Then
\begin{eqnarray}\label{eq:Weakmono}
   & &   \rho^{1-n}\mathcal{F}_S\left(\{E_j\},B_\rho\right) + C\rho
      +  \sum_{j = 0}^2\frac{\alpha_j}{8}\int_{B_r\setminus B_\rho} |x|^{1-n} 
             \left\langle \frac{x}{|x|}, \frac{D\chi_{E_j}}{|D\chi_{E_j}|} \right\rangle^4 |D\chi_{E_j}|\, dx \nonumber\\
       & & \hspace{.8in} \leq r^{1-n}\mathcal{F}_S\left(\{E_j\},B_r\right) + Cr .\quad\qquad
\end{eqnarray}
\end{corollary}

%%%%%%%%%%%%%%%%%%%%%%%%%%%%%%%%%%%%%%%%%%%%%%%%

\section{The Monotonicity of Scaled Energy (Part II)}
\label{MSE2}

In this section we temporarily abandon the volume constraint and produce a sharp formula
for monotonicity of scaled energy.
%Now that we have established a monotonicity formula for our constrained problem, we note
%that if we temporarily abandon the volume constraint, then we can get a sharp formula for
%the monotonicity of the scaled energy insofar as we can produce a formula with equality.
\begin{theorem}  \label{sharpmono}
Let $d = dist(0,\partial \Omega)$.  If $\{E_j\}$ is a D-minimizer in $\Omega$ and
%$\mathcal{F}_S$ in $\Omega$ and
$0\in\Omega\cap(\cap_j\partial E_j)$, then for a.e. $r \in (0,d)$, 
\begin{equation}
\label{eq:Mono}
       \frac{d}{dr} \left(r^{1-n} \cdot \mathcal{F}_S(\{E_j\},B_r) \right)
%\frac{d}{dr}\frac{\mathcal{F}_S(\{E_j\},B_r)}{r^{n-1}}
    = \frac{d}{dr}\sum^2_{j=0} \alpha_j \int_{B_r\cap\partial^*\!E_j} \frac{(\nu_{E_j}(x)\cdot x)^2}{|x|^{n+1}}\,
                    d\mathcal{H}^{n-1}(x).
\end{equation}
\end{theorem}

\begin{proof}
We follow Theorem 28.9 of Maggi \cite{Mag}.
Given any $\varphi\in C^\infty(\R;[0,1])$ with $\varphi = 1$ on $(-\infty, 1/2)$,
$\varphi = 0$ on $(1,\infty)$ and $\varphi^\prime \leq 0$ on $\R$, we define the following associated functions
\begin{equation}\label{def:Phi}
                     \Phi(r) 
                 = \sum_{j=0}^2\alpha_j\int_{\partial^*\! E_j} \varphi\left(\frac{|x|}{r}\right)\, 
                                          d\mathcal{H}^{n-1}(x), \quad r\in(0,d),
\end{equation}
and
\begin{equation}\label{def:Psi}
                     \Psi(r)
                  = \sum_{j=0}^2\alpha_j\int_{\partial^*\! E_j} \varphi\left(\frac{|x|}{r}\right) 
                                          \frac{(x\cdot\nu_{E_j}(x))^2}{|x|^2}\, d\mathcal{H}^{n-1}(x), \quad r\in(0,d).
\end{equation}
Note 
\begin{equation}
                       \Phi^\prime(r)
                   = -\sum_{j=0}^2\alpha_j\int_{\partial^* \!E_j} \varphi^\prime\left(\frac{|x|}{r}\right)
                                          \frac{|x|}{r^2}\, d\mathcal{H}^{n-1}(x), \quad r\in(0,d),
\end{equation}
and
\begin{equation}
                        \Psi^\prime(r)
                   = -\sum_{j=0}^2\alpha_j\int_{\partial^*\! E_j} \varphi^\prime\left(\frac{|x|}{r}\right)
                                          \frac{|x|}{r^2} \frac{(x\cdot\nu_{E_j}(x))^2}{|x|^2}\,
                                          d\mathcal{H}^{n-1}(x), \quad r\in(0,d).
\end{equation}
Define
\begin{equation}
                                  T_r\in C^1_c(\Omega;\R^n), 
                  \quad       \mathit{T_r(x) = \varphi\left( \frac{|x|}{r}\right) x, \quad x\in\R^n},
\end{equation}
and observe the identities
\begin{eqnarray}
\nabla T_r &
       =& \varphi\left( \frac{|x|}{r}\right)Id
         + \frac{|x|}{r}\varphi^\prime\left( \frac{|x|}{r}\right) \frac{x}{|x|}\otimes\frac{x}{|x|}, 
                                               \quad \forall x \in \R^n \\
\text{div} \, T_r &
        =& n\varphi\left( \frac{|x|}{r}\right)
         + \frac{|x|}{r}\varphi^\prime\left( \frac{|x|}{r}\right), 
                                               \quad \forall x \in \R^n \\
\nu_E \cdot \nabla T_r\nu_E &
         =& \varphi\left( \frac{|x|}{r}\right)
         + \frac{|x|}{r}\varphi^\prime\left( \frac{|x|}{r}\right) \frac{(x\cdot\nu_E(x))^2}{|x|^2},
                                               \quad \forall x \in \partial^*E \\
\text{div}_E\, T_r &
          =& \text{div}\, T_r - \nu_E \cdot \nabla T_r\nu_E  \nonumber\\
        &=& (n-1)\varphi\left( \frac{|x|}{r}\right) 
          + \frac{|x|}{r}\varphi^\prime\left( \frac{|x|}{r}\right)\left( 1 - \frac{(x\cdot\nu_E(x))^2}{|x|^2} \right) \;.
\end{eqnarray}
Now we quote Theorem 17.5 of \cite{Mag}:
\begin{theorem}[First variation of perimeter]   \label{FVP}
Suppose $A \subset \R^n$ is open, $E$ is a set of locally finite perimeter, and
$\{ f_t \}_{|t| < \epsilon}$ is a local variation in $A.$  Then
\begin{equation}
   \int_{A} | D\chisub{f_t(E)}| = \int_{A} | D\chisub{E}| 
                                                + t \int_{\partial^*\!E} \mathrm{div}_{E} T \, d\mathcal{H}^{n-1}(x)
                                                + O(t^2) \, ,
\label{eq:FVP}
\end{equation}
where $T$ is the initial velocity of $\{ f_t \}_{|t| < \epsilon}$ and $\mathrm{div}_{E} T: \partial^*\!E \rightarrow \R$
is given above.  ($T$ is the initial velocity of $\{ f_t \}_{|t| < \epsilon}$ means
$$\frac{\partial}{\partial t} f(t,x) = T(x)$$ when $f$ is evaluated at $t = 0.$)
\end{theorem}
In the same way that Maggi proves Corollary 17.14 from this statement, we can show:
\begin{corollary}[Vanishing sums of mean curvature]   \label{VSMC}
A permissible triple $\{E_j\}$ is stationary for $\mathcal{F}_S$ in $\Omega$ if and only if
\begin{equation}
     \sum_{j=0}^2 \alpha_j \int_{\partial^*\!E_j} \mathrm{div}_{E_j} T \, d\mathcal{H}^{n-1}(x) = 0 \, ,
     \quad \forall T \in C^1_{c}(\Omega; \R^n).
\label{eq:Edivrel}
\end{equation}
\end{corollary}
Returning to our proof of Theorem \ref{sharpmono} we compute:
\begin{eqnarray}
(n-1)\Phi(r) - r\Phi^\prime(r) &
            =& (n-1) \sum_{j=0}^2\alpha_j \int_{\partial^* \!E_j} \varphi\left(\frac{|x|}{r}\right)\, 
                              d\mathcal{H}^{n-1}(x) \nonumber\\
          & & + \sum_{j=0}^2\alpha_j\int_{\partial^* \!E_j} \varphi^\prime\left(\frac{|x|}{r}\right) \frac{|x|}{r}\,
                              d\mathcal{H}^{n-1}(x) \\ 
           &=&  \sum_{j=0}^2\alpha_j\int_{\partial^*\! E_j} \varphi^\prime\left(\frac{|x|}{r}\right) \frac{|x|}{r} \cdot
                              \frac{(x\cdot\nu_{E_j}(x))^2}{|x|^2}\,
                              d\mathcal{H}^{n-1}(x) \\
           &=& - r\Psi^\prime(r),
\end{eqnarray}
or
\begin{equation}
\frac{\Phi^\prime(r)}{r^{n-1}} - (n-1)\frac{\Phi(r)}{r^{n}} = \frac{\Psi^\prime(r)}{r^{n-1}} \quad a.e.\ r\in(0,d).
\end{equation}
Next, for $\epsilon\in(0,1)$, define Lipschitz functions $\varphi_\epsilon : \R \rightarrow [0,1]$ as
\begin{equation}
        \varphi_\epsilon(s) 
                   = \chi_{(-\infty,1-\epsilon)}(s) + \frac{1-s}{\epsilon}\chi_{(1-\epsilon, 1)}(s), \quad s\in\R,
\end{equation}
and define $\Phi_\epsilon(r)$ and $\Psi_\epsilon(r)$ by replacing $\varphi$ with $\varphi_{\epsilon}$ in the
definitions of $\Phi(r)$ and $\Psi(r)$ respectively.
Then, by approximation using Theorem~2, p.~\!\!\!~172 of \cite{EG} or something similar, for
$\epsilon\in(0,1)$ and $\varphi = \varphi_\epsilon$ in \eqref{def:Phi} and \eqref{def:Psi} we obtain
\begin{equation}\label{eqn:EpsEqn}
   \frac{\Phi^\prime_\epsilon(r)}{r^{n-1}} - (n-1)\frac{\Phi_\epsilon(r)}{r^{n}} = 
   \frac{\Psi^\prime_\epsilon(r)}{r^{n-1}} \quad a.e.\ r\in(0,d).
\end{equation}
Define $\Phi_0(r) = \mathcal{F}_S(\{E_j\}, B_r)$ and 
\begin{equation}
               \gamma(r) 
              = \sum_{j=0}^2\alpha_j\int_{B_r\cap\partial^*\!E_j} \frac{(\nu_{E_j}(x)\cdot x)^2}{|x|^{n+1}}\,
                                               d\mathcal{H}^{n-1}(x), \quad r\in(0,d).
\end{equation}
For $r \in (0,d),$ the Lebesgue Dominated Convergence Theorem implies 
\begin{equation}\label{eqn:compact}
                        \Phi_\epsilon 
                 \rightarrow  \sum_{j=0}^2\alpha_j\int_{\partial^* \!E_j\cap B_r}  d\mathcal{H}^{n-1}(x)
                = \Phi_0, \quad \mbox{ as $\epsilon \rightarrow 0$}.
\end{equation}
\begin{claim}
For $a.e.\, r\in(0,d)$ 
\begin{equation}\label{eqn:claim}
                      \Phi^\prime_\epsilon(r) \rightarrow \Phi^\prime_0(r), \quad
                      \Psi^\prime_\epsilon(r) \rightarrow r^{n-1}\gamma^\prime(r),
\end{equation}
as $\epsilon\rightarrow 0$. In particular, this holds for every $r\in(0,d)$ where $\Phi_0$ and $\gamma$ are differentiable.
\end{claim}
\begin{proof} (of claim)
Upon examining, we write
\begin{equation}
\Phi_\epsilon(r) = 
       \sum^2_{j=0} \alpha_j \left(
             \int_{\partial^*\!E_j\cap B_{r(1-\epsilon)}} \, d\mathcal{H}^{n-1}(x) +
             \int_{\partial^*\!E_j\cap(B_r \setminus B_{r(1-\epsilon)})} \left(\frac{1}{\epsilon} - \frac{|x|}{\epsilon r}\right)\, d\mathcal{H}^{n-1}(x) \right).
\end{equation}
We wish to differentiate the term in the parentheses above.  We can express that term as:
\begin{equation}
       I_1(r) + I_2(r) :=
       \int_{\partial^*\!E_j\cap B_{r(1-\epsilon)}} \left(1 - \frac{1}{\epsilon} + \frac{|x|}{\epsilon r}\right)\, d\mathcal{H}^{n-1}(x) +
       \int_{\partial^*\!E_j\cap B_r} \left(\frac{1}{\epsilon} - \frac{|x|}{\epsilon r}\right)\, d\mathcal{H}^{n-1}(x) \;.
\label{eq:twoints}
\end{equation}
Then
\begin{eqnarray}
I_1^{\prime}(r) &=& \int_{\partial^*\!E_j\cap \, \partial B_{r(1-\epsilon)}} 0 \, d\mathcal{H}^{n-1}(x) 
                           - \int_{\partial^*\!E_j\cap \, B_{r(1-\epsilon)}} \left( \frac{|x|}{\epsilon r^2}\right) d\mathcal{H}^{n-1}(x) \nonumber \\
                &=& - \int_{\partial^*\!E_j\cap \, B_{r(1-\epsilon)}} \left( \frac{|x|}{\epsilon r^2}\right) d\mathcal{H}^{n-1}(x) \, ,
\label{eq:I1der}
\end{eqnarray}
and
\begin{eqnarray}
I_2^{\prime}(r) &=& \int_{\partial^*\!E_j\cap \, \partial B_{r}} 0 \, d\mathcal{H}^{n-1}(x) 
                           + \int_{\partial^*\!E_j\cap \, B_{r}} \left( \frac{|x|}{\epsilon r^2}\right) d\mathcal{H}^{n-1}(x) \nonumber\\
                &=&  \int_{\partial^*\!E_j\cap \, B_{r}} \left( \frac{|x|}{\epsilon r^2}\right) d\mathcal{H}^{n-1}(x)\, .
\label{eq:I2der}
\end{eqnarray}
Then it follows that
\begin{equation}
\Phi^\prime_\epsilon(r) = \frac{1}{\epsilon r}\sum^2_{j=0}
         \alpha_j\int_{\partial^*\!E_j\cap(B_r \setminus B_{r(1-\epsilon)})}
          \frac{|x|}{r}\, d\mathcal{H}^{n-1}(x), \quad a.e.\, r\in(0,d),
\end{equation}
and we estimate to obtain
\begin{equation}
(1-\epsilon)\frac{\mathcal{F}_S(\{E_j\},B_r) - \mathcal{F}_S(\{E_j\},B_{r-\epsilon r})}{\epsilon r} \leq \Phi^\prime_\epsilon(r) \leq \frac{\mathcal{F}_S(\{E_j\},B_r) - \mathcal{F}_S(\{E_j\},B_{r-\epsilon r})}{\epsilon r}. 
\end{equation}
Thus, if $\Phi_0(r)$ is differentiable at $r$, then $\Phi^\prime_\epsilon(r) \rightarrow \Phi^\prime_0(r)$ as $\epsilon\rightarrow 0^+$.

Next, upon examining $\Psi_\epsilon$, we write
\begin{equation}
\begin{array}{ll}
\Psi_\epsilon(r) &\!\!\!\!\displaystyle{= 
       \sum^2_{j=0} \alpha_j \left(
                 \int_{\partial^*\!E_j\cap B_{r(1-\epsilon)}} 
                                \frac{(x\cdot \nu_{E_j}(x))^2}{|x|^2} \, d\mathcal{H}^{n-1}(x) \right.} \\
             \ &+ \displaystyle{\left. \int_{\partial^*\!E_j\cap(B_r \setminus B_{r(1-\epsilon)})} 
                                \left(\frac{1}{\epsilon} - \frac{|x|}{\epsilon r}\right) \frac{(x\cdot \nu_{E_j}(x))^2}{|x|^2}  \,
                                                d\mathcal{H}^{n-1}(x) \right)} \;.
\end{array}
\end{equation}
Once again we wish to differentiate this term, so we express the term within the parentheses as:
\begin{equation}
\begin{array}{rl}
       \tilde{I}_1(r) + \tilde{I}_2(r) 
      &\displaystyle{\!\!\!\!\!:=
       \int_{\partial^*\!E_j\cap B_{r(1-\epsilon)}} \left(1 - \frac{1}{\epsilon} + \frac{|x|}{\epsilon r}\right)\, 
                                  \frac{(x\cdot \nu_{E_j}(x))^2}{|x|^2} \, d\mathcal{H}^{n-1}(x)} \\
    \ &\displaystyle{+
       \int_{\partial^*\!E_j\cap B_r} \left(\frac{1}{\epsilon} - \frac{|x|}{\epsilon r}\right)\, 
                                 \frac{(x\cdot \nu_{E_j}(x))^2}{|x|^2} \, d\mathcal{H}^{n-1}(x) \;.}
\end{array}
\label{eq:twointsagain}
\end{equation}
Then
\begin{eqnarray}
\tilde{I}_1^{\prime}(r) &=& \int_{\partial^*\!E_j\cap \, \partial B_{r(1-\epsilon)}} 0 \, d\mathcal{H}^{n-1}(x) 
                           - \int_{\partial^*\!E_j\cap \, B_{r(1-\epsilon)}} \left( \frac{|x|}{\epsilon r^2}\right) 
                                                  \frac{(x\cdot \nu_{E_j}(x))^2}{|x|^2} d\mathcal{H}^{n-1}(x) \nonumber\\
                        &=& - \int_{\partial^*\!E_j\cap \, B_{r(1-\epsilon)}} \left( \frac{|x|}{\epsilon r^2}\right) 
                        \frac{(x\cdot \nu_{E_j}(x))^2}{|x|^2} d\mathcal{H}^{n-1}(x) \, ,
\label{eq:I1tilder}
\end{eqnarray}
and
\begin{eqnarray}
\tilde{I}_2^{\prime}(r) &=& \int_{\partial^*\!E_j\cap \, \partial B_{r}} 0 \, d\mathcal{H}^{n-1}(x) 
                           + \int_{\partial^*\!E_j\cap \, B_{r}} \left( \frac{|x|}{\epsilon r^2}\right)
                                                  \frac{(x\cdot \nu_{E_j}(x))^2}{|x|^2} d\mathcal{H}^{n-1}(x) \nonumber\\
                        &=& \int_{\partial^*\!E_j\cap \, B_{r}} \left( \frac{|x|}{\epsilon r^2}\right)
                        \frac{(x\cdot \nu_{E_j}(x))^2}{|x|^2} d\mathcal{H}^{n-1}(x)\, ,
\label{eq:I2tilder}
\end{eqnarray}
implying
\begin{equation}
             \frac{\Psi^\prime_\epsilon(r)}{r^{n-1}} 
                        = \frac{1}{\epsilon r} \sum^2_{j=0}\alpha_j 
               \int_{\partial^*\!E_j\cap(B_r \setminus B_{r(1-\epsilon)})} \left( \frac{|x|}{r} \right)^n 
                            \frac{(x\cdot \nu_{E_j}(x))^2}{|x|^{n+1}}  \,  d\mathcal{H}^{n-1}(x) \quad a.e.\, r\in(0,d).
\end{equation}
As before, it follows that 
\begin{equation}
(1-\epsilon)^n\frac{\gamma(r) - \gamma(r-\epsilon r)}{\epsilon r} \leq 
        \frac{\Psi^\prime_\epsilon(r)}{r^{n-1}} \leq \frac{\gamma(r) - \gamma(r-\epsilon r)}{\epsilon r},
\end{equation}
and if $\gamma(r)$ is differentiable at $r$, then $\Psi^\prime_\epsilon(r) \rightarrow r^{n-1}\gamma^\prime_0(r)$ as $\epsilon\rightarrow 0^+$.
Therefore the claim holds.
\end{proof}
From \eqref{eqn:EpsEqn}, \eqref{eqn:compact} and \eqref{eqn:claim} we find
\begin{equation}
\frac{\Phi^\prime_0(r)}{r^{n-1}} - (n-1)\frac{\Phi_0(r)}{r^{n}} = \gamma^\prime(r) \, ,
\end{equation}
and this proves Equation\refeqn{Mono}\!\!.
\end{proof}

As we have mentioned, there are other related monotonicity formulas.  Our first monotonicity
formula\refeqn{Weakmono}is based off of work found in Guisti's monograph, however, we
sharpened it by including an explicit increment in the difference in the scaled energies for
two different radii.  This increment is measuring how far a configuration deviates from a
cone.  Maggi completely characterized the monotonicity for the problem that Guisti considered,
insofar as he produced a formula with an equality, and his scaled energy is constant as a function
of radius when the configuration is a cone.  We based our second approach on his methods,
and our generalization is found in Theorem\refthm{sharpmono}\!\!, although like Maggi we do not
consider a volume constraint in obtaining this result.

Morgan \cite{Mo2016} defines a mass ratio $\Theta(T,a,r)$ that is equivalent to Guisti's
formulation of scaled energy (which is the formulation used in this paper). Then Morgan
goes on to prove a monotonicity result (credited to Federer \cite{Fed}) saying that
$\Theta(T,a,r)$ is a monotonically increasing function of $r$.  This result corresponds to
what Guisti and Maggi wrote about, but is apparantly not as sharp as Maggi's result.
On page 108 of \cite{Mo2016}, Morgan describes Allard's results \cite{Al} in that
``{\em integral varifolds of bounded first variation include surfaces of constant or
bounded mean curvature and soap bubble clusters.  They satisfy a weakened versions
of the monotonicity ... the area ratio times $e^{Cr}$ is monotonically increasing, where
$C$ is a bound on the first variation or mean curvature}'' [emphasis in original].
Because the value of $C$ can be taken to be zero in the case where there is no volume
constraint, we have reproduced, but not improved on this result.  In our formula in the case
with no volume constraint, the derivative of our scaled energy is given as an explicit positive
function which shows exactly how much the energy increases as the radius increases.

%To obtain our results, we scale away the gravity terms and observe that we may as well
%discard them.  Then we treat the remaining surface terms of the energy
%directly.

%Although the Allard formula applies far
%more generally, we have an explicit gain after the addition of only a quadratic term which appeared
%because of Almgren's Volume Adjustment Lemma.

%In this last section insofar as we considered the situation without
%the volume constraint, we did not need the Almgren Lemma and our formula is sharp as the derivative of our scaled energy
%is given as an explicit positive function which shows exactly how much the energy increases as the radius increases.

%In
%terms of the usual applications of Monotonicity formulas we are not aware of any improvement that we can get from our
%more sharp formula.  On the other hand, as a simple example of a somewhat unusual application we can observe the
%following.

In both sections 5 and 6, an additional result that we were unable to prove was of the uniqueness of the blowup limit.
The introduction of Almgren's Big Regularity Paper \cite{Am2000}
discusses this difficulty, and some examples of slowly rotating configurations are in
Leonardi \cite{L2002}.  In fact, Leonardi gives an example which is spiral but which always blows up to the same
conical formation.  (See \cite{L2002}[Example 4.7].  This sort of behavior (i.e. a unique type of blowup limit, but no
uniqueness of the limit because of the necessity to get a convergent subsequence) can also be found in a paper by
the first author \cite{B}.)  In one related setting there has been success in showing that
the tangent cone is unique: See White \cite{W1983}.   To summarize, although we eventually have specific angle
conditions satisfied by the blowup limits, we cannot prove that the actual minimizers do not have some rotation that
becomes slower and slower that prevents the existence of multiple blowup limits.  (We do certainly conjecture that
the blowup limit will be unique.)

\section{Minimal Cones}
\label{MC}

We begin with the following result estimating the minimal energies by their Dirichlet data.
\begin{lemma}[Extension of Lemma 5.6 of \cite{G}] \label{G5.6}
Suppose $\{E_j\}$ and $\{\hat{E}_j\}$ are V-permissible for the 
same volume constraints in $B_R$ and are identical in $B_{\rho}^c.$  Suppose further that 
$\rho$ is small enough to guarantee that any perturbation to $\{E_j\}$ or to $\{\hat{E}_j\}$ within $B_{\rho}$
gives us something to which Almgren's Volume Adjustment Lemma applies.  (See Lemma~\ref{PfVA}.)  Then
\begin{equation}
        |\newnu_V (\{E_j\}, \rho) - \newnu_V (\{\hat{E}_j\}, \rho)|
   \leq \sum_{j = 0}^{2} \alpha_{j} \int_{\partial B_{\rho}} |\chisub{E_j}^{-} - \chisub{\hat{E}_j}^{-}| \; d\mathcal{H}^{n-1} + C\sum_{j=0}^2|\Delta V_j| \;,
\label{eq:G5.4}
\end{equation}
where $\Delta V_i$ is the symmetric difference $E_i \Delta \hat{E}_i.$  If instead of ``V-permissible''
we have ``permissible,'' then for any positive $\rho < R$ we have:
\begin{equation}
        |\newnu (\{E_j\}, \rho) - \newnu (\{\hat{E}_j\}, \rho)|
   \leq \sum_{j = 0}^{2} \alpha_{j} \int_{\partial B_{\rho}} |\chisub{E_j}^{-} - \chisub{\hat{E}_j}^{-}| \; d\mathcal{H}^{n-1} \;.
\label{eq:G5.42}
\end{equation}
\end{lemma}
\begin{proof}
The proof for Equation\refeqn{G5.42}is almost identical to the proof for Equation\refeqn{G5.4}\!, but
it is a little bit easier, so we will only prove Equation\refeqn{G5.4}\!.
Given $\epsilon > 0,$ we can choose $\{E_j^{\varphi}\}$ V-permissible so that we satisfy two relations:
\begin{enumerate}
   \item $\text{spt}(\chisub{E_j^{\varphi}} - \chisub{E_j}) \subset \subset B_{\rho},$ and
   \item $\mathcal{F}_S(\{E_j^{\varphi}\}, \rho) \leq \epsilon + \newnu_V (\{E_j\}, \rho) \;.$
\end{enumerate}
Let $\rho_k \uparrow \rho$ be taken such that
$$\int_{\partial B_{\rho_k}} |D\chisub{E_j}| = \int_{\partial B_{\rho_k}} |D\chisub{\hat{E}_j}| = 0 \;,$$
and $\text{spt}(\chisub{E_j^{\varphi}} - \chisub{E_j}) \subset \subset B_{\rho_k} \;$ for all $k \in \N$ and $j \in \{0,1,2\}.$
For any $j$ we define the set $\hat{E}_{j,k}$ by taking the union of $B_{\rho_k} \cap E_j^{\varphi}$ and
$\{ B_R \setminus B_{\rho_k} \} \cap \hat{E}_j.$  Now observe that $\{\hat{E}_{j,k}\}$ is permissible up to the volume constraint violation.  We then use Almgren's Lemma~\ref{PfVA} in order to compute:
\begin{alignat*}{1}
  \newnu_V (\{\hat{E}_j\}, \rho) &\leq \sum_{j = 0}^2 \alpha_j \int_{B_{\rho}} |D\chisub{\hat{E}_{j,k}}| 
								+ C\sum_{j=0}^2|\Delta V_j|\\
                           &= \sum_{j = 0}^2 \alpha_j \left( \int_{B_{\rho_k}} |D\chisub{E_j^{\varphi}}| 
                                                           + \int_{B_R \setminus B_{\rho_k}} |D\chisub{\hat{E}_j}| 
                                                           + \int_{\partial B_{\rho_k}} |\chisub{E_j} - \chisub{\hat{E}_j}| \right) 
								+ C\sum_{j=0}^2|\Delta V_j|\\
                           &\leq \sum_{j = 0}^2 \alpha_j \left( \int_{B_{\rho}} |D\chisub{E_j^{\varphi}}|
                                                              + \int_{B_R \setminus B_{\rho_k}} |D\chisub{\hat{E}_j}|
                                                              + \int_{\partial B_{\rho_k}} |\chisub{E_j} - \chisub{\hat{E}_j}| \right) 										+ C\sum_{j=0}^2|\Delta V_j|\\
                           &\leq \epsilon + \newnu_V (\{E_j\}, \rho) 
                                          + \sum_{j = 0}^2 \alpha_j \left(\int_{B_R \setminus B_{\rho_k}} |D\chisub{\hat{E}_j}|
                                                                        + \int_{\partial B_{\rho_k}} |\chisub{E_j} - \chisub{\hat{E}_j}| \right) 
									+ C\sum_{j=0}^2|\Delta V_j|\\
                           &\rightarrow \epsilon + \newnu_V (\{E_j\}, \rho) 
             + \sum_{j = 0}^2 \alpha_j \left( \int_{\partial B_{\rho}} |\chisub{E_j}^{-} - \chisub{\hat{E}_j}^{-}| \right) + C\sum_{j=0}^2|\Delta V_j|\;.
\end{alignat*}
Now by using the fact that $\epsilon > 0$ is arbitrary and by the symmetry of the equation
that we are trying to prove, we are done.
\end{proof}

\begin{lemma}[Analogue of Lemma 9.1 of \cite{G}] \label{lemma:MC}
Let $\Omega\subset\R^n$ be open, let $\{E_{j,k}\}$ be a sequence of sets that DV-minimize $\mathcal{F}_S$ over $\Omega$ 
i.e. $\{E_{j,k}\}$ are taken such that 
\begin{equation}
   \Psi_V(\{E_{j,k}\},A) = 0\, , \quad \forall A \subset\subset \Omega
\end{equation}
(with potentially different Dirichlet data and volume constraints for each $k$).
Suppose there exists a triple $\{E_j\}$ such that 
\begin{equation}
   \chi_{E_{j,k}}\rightarrow \chi_{E_j} \quad  \mbox{ in } L^1_{loc}(\Omega), \quad j=0,1,2
\label{eq:L1Conv}
\end{equation}
Then $\{E_j\}$ is a DV-minimizer of $\mathcal{F}_S$ over $\Omega$:
\begin{equation}\label{eqn:E_j_min}
   \Psi_V (\{E_j\},A) = 0 \, , \quad \forall A \subset\subset \Omega.
\end{equation}
Moreover, if $L \subset\subset \Omega$ is any open set such that 
\begin{equation}
   \int_{\partial L}|D\chi_{E_j}| = 0, \quad j = 0,1,2,
\end{equation}
then we have
\begin{equation}
   \lim_{k\rightarrow\infty} \mathcal{F}_S(\{E_{j,k}\},L) = \mathcal{F}_S(\{E_j\},L).
% \lim_{k\rightarrow\infty} \int_L |D\chi_{E_{j,k}}| = \int_L |D\chi_{E_j}| \quad j = 0,1,2.
\end{equation}
\end{lemma}
\begin{remark}[Weakness of some of the hypotheses]  \label{WeakHyp}
Equation\refeqn{L1Conv}can be guaranteed by Helly's Selection Theorem as long as all of the configurations
have uniformly bounded energy.
\end{remark}
\begin{proof}
Let $A\subset\subset \Omega$.  We may suppose that $\partial A$ is smooth, so that for every $k$:
\begin{equation}
   \mathcal{F}_S(\{E_{j,k}\},A) \leq \left(\mathcal{H}^{n-1}(\partial A) 
                      + \frac{1}{2} \omega_{n-1} \left(\frac{\text{ max diam of } A}{2}\right)^{n-1}\right)
                          \sum_{j=0}^2 \alpha_j,
\end{equation}
which follows by covering $\partial A$ with all three values, and bounding the minimal  energy of
$\{E_j\}$ by a (standard) competitor on a possibly larger domain.
Then lower semicontinuity implies the same inequality holds with $\{E_{j,k}\}$ replaced with $\{E_j\}$.

For $t>0$, let
\begin{equation}
   A_t = \{x\in\Omega : \text{dist}(x,A) < t\}.
\end{equation}
We have
\begin{equation}
   \lim_{k\rightarrow\infty} \int_{A_t} |\chi_{E_{j,k}} - \chi_{E_j}|\, dx = 0, \quad j = 0,1,2
\end{equation}
and therefore there exists a subsequence $\{E_{j,k_i}\}$ such that for almost every $t$ close to 0
\begin{equation}
   \lim_{k\rightarrow\infty}\int_{\partial A_t} |\chi_{E_{j,k_i}} - \chi_{E_j}|\, d\mathcal{H}^{n-1} = 0, \quad j = 0,1,2.
\end{equation}
From Lemma~\ref{G5.6}, as $\sum_{j=0}^2|\Delta V_{j,k}|\rightarrow 0$, we have for those $t$
\begin{equation}
   \lim_{k\rightarrow\infty} \newnu_V (\{E_{j,k_i},A_t) = \newnu_V (\{E_j,A_t)
\end{equation}
and by lower semicontinuity
\begin{equation}
   \Psi_V (\{E_j\},A_t) = 0.
\end{equation}
Thus \eqref{eqn:E_j_min} holds.  Now let $L\subset\subset \Omega$ be such that
$\int_{\partial L} |D\chi_{E_j}| = 0$ for $j= 0,1,2$, and let $A$ be a smooth open
set such that $L\subset\subset A\subset\subset \Omega$.  Let $\{F_{j,k}\}$ be any
subsequence of $\{E_{j,k}\}$.  Repeating the same argument as above, there is a
set $A_t$ and a subsequence $\{F_{j,k_i}\}$  such that 
\begin{equation}
    \lim_{k\rightarrow\infty}\newnu_V (\{F_{j,k_i}\},A_t) = \newnu_V (\{F_j\},A_t).
\end{equation}
Since $\Psi_V (\{F_{j,k_i}\},A_t) = \Psi_V (\{E_j\},A_t) = 0$ we have 
\begin{equation}
    \lim_{k\rightarrow\infty}\mathcal{F}_S(\{F_{j,k_i}\},A_t) = \mathcal{F}_S(\{F_j\},A_t),
\end{equation}
thus from Theorem~\ref{DT2} (with $A_t$ playing the role of $\Omega$)
\begin{equation}
    \lim_{k\rightarrow\infty} \mathcal{F}_S(\{F_{j,k}\},L) = \mathcal{F}_S(\{E_j\},L),
\end{equation}
completing the proof.
\end{proof}

\begin{figure}[!h]
	\centering
	\scalebox{.55}{\includegraphics{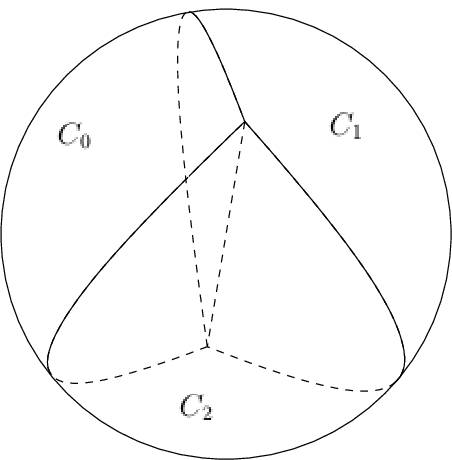}}
	\caption{Limiting configuration as cones.}
	\label{fig:Cones}
\end{figure}
\begin{theorem}[Analogue of Theorem 9.3 of \cite{G}]   \label{thm:cones}
Suppose $\{E_j\}$ is a V-minimizer of $\mathcal{F}_S$ in $B_1$,
that is $\Psi_V (\{E_j\},B_1)=0$, such that $0\in\partial E_0 \cap \partial E_1 \cap \partial E_2$.  For each $t>0$, let
\begin{equation}
      E_{j,t} = \{ x \in \R^n : \mathit{tx} \in \mathit{E_j} \}, \quad j = 0,1,2.
\end{equation}
Then for every sequence $\{t_i\}$ tending to zero there exists a subsequence $\{s_i\}$
such that $E_{j,s_i}$ converges locally in $\R^n$ to permissible sets $C_j$. 
Moreover, $\{C_j\}$ are cones with positive density at the origin (the vertex of the cones) satisfying 
\begin{equation}
     \Psi_V (\{C_j\},A) = 0 \quad \forall A \subset\subset \R^n.
\end{equation}
\end{theorem}
See Figure~\ref{fig:Cones}.  
Note that this is a tiny bit more than an analogue of Theorem 9.3 of \cite{G} because we can use the Elimination
Theorem to get the statement about the positive density of the cones at the origin.
In light of this result, we define a triple $\{E_j\}$ to V-minimize $\mathcal{F}_S$ over
$\R^n$ if 
\begin{equation}
\Psi_V (\{E_j\},A) = 0 \quad \forall A \subset\subset \R^n.
\end{equation}

\begin{proof}
Let $t_i\rightarrow 0$.  The first step is to show that for every $R>0$ there exists a subsequence
$\{\sigma_i \}$ such that $\{E_{j,\sigma_i}\}$ converges in $B_R$.  We have
\begin{equation}
   \mathcal{F}_S(\{E_{j,t}\},B_R) = t^{1-n}\mathcal{F}_S(\{E_j\},B_{Rt})
\end{equation}
and so choosing $t$ sufficiently small  (so that $Rt<1$) we have that $E_{j,t}$ is a V-minimizer of $\mathcal{F}_S$ over $B_R$ and 
\begin{equation}\label{eqn:eng_bnd}
        \mathcal{F}_S(\{E_{j,t}\},B_R) = t^{1-n}\mathcal{F}_S(\{E_j\},B_{Rt}) 
   <  \left(\mathcal{H}^{n-1}(\partial B_1) + \frac{1}{2} \omega_{n-1} \right) R^{n-1} \sum_{j=0}^2 \alpha_j.
\end{equation}
Hence, by Helly's Selection Theorem (see Theorem\refthm{Helly}\!), a subsequence $\{E_{j,\sigma_i}\}$
converges to the triple of sets $\{C_{j,R}\}$ in $B_R$.  Taking a sequence
$R_t\rightarrow \infty$ we obtain, by a diagonal process, the triple of sets
$\{C_j\}\subseteq\R^n$ and a sequence $\{s_i\}$ such that $\{E_{j,s_i}\} \rightarrow \{C_j\}$
locally.  Now,  applying Lemma~\ref{lemma:MC}, we see that $\{C_j\}$ is a V-minimizer of $\mathcal{F}_S$ over
$\R^n$ in the sense that
\begin{equation}
   \Psi_V(\{C_j\},A) = 0 \quad \forall A \subset\subset \R^n.
\end{equation}
The positive density of the $C_j$ at the origin follows immediately by applying the Elimination Theorem.
(See Theorem \ref{LET}.)  If we assume the opposite, then we can use the Elimination Theorem to show that
$0$ was not a triple point at the outset.
It remains to show that the $C_j$ are cones.

By Lemma~\ref{lemma:MC} we have that, for almost all $R>0$,
\begin{equation}
   \mathcal{F}_S(\{E_{j,s_i}\},B_R) \rightarrow \mathcal{F}_S(\{C_j\},B_R).
\end{equation}
Hence, if we define 
\begin{equation}
   p(t) = t^{1-n}\mathcal{F}_S(\{E_j\},B_t)  + Ct= \mathcal{F}_S(\{E_{j,t}\},B_1) + Ct,
\end{equation}
where $C$ is the constant from Almgren's Volume Adjustment Lemma (see Lemma~\ref{PfVA}), we have,
for almost all $R>0$,
\begin{equation}
   \lim_{i\rightarrow\infty} p(s_iR) = R^{1-n}\mathcal{F}_S(\{C_j\},B_R),
\end{equation}
as $i \rightarrow \infty.$  (We must have $s_i \rightarrow 0$ as $i \rightarrow \infty.$)
Also, from Equation \eqref{eq:Weakmono}, $p(t)$ is increasing in $t.$

If $\rho < R$, then for every $i$ there exists an $m_i>0$ such that 
\begin{equation}
    s_i\rho > s_{i+m_i}R.
\end{equation}
Then
\begin{equation}
    p(s_{i+m_i}R)\leq p(s_i\rho) \leq p(s_iR)
\end{equation}
so that 
\begin{equation}\label{eqn:lim_p}
    \lim_{i\rightarrow\infty} p(s_i\rho) = \lim_{i\rightarrow\infty} p(s_iR) = R^{1-n}\mathcal{F}_S(\{C_j\},B_R)
\end{equation}
Thus we have proved that
\begin{equation}
    \rho^{1-n}\mathcal{F}_S(\{C_j\},B_\rho)
\end{equation}
is independent of $\rho$, and so from Lemma~\ref{G5.3}
we have
\begin{eqnarray}
                 \sum^2_{j=0}\alpha_j\int_{\partial B_1} |\chi^{-}_{C_j}(rx) - \chi^{-}_{C_j}(\rho x)| \; 
                                         d\mathcal{H}^{n-1} 
     &\leq&  \sum^2_{j=0}\alpha_j\int_{B_r \setminus B_\rho} 
                                         \left| \left\langle \frac{x}{|x|^n}, D\chi_{C_j} \right\rangle \right| \\
     &\leq&   r^{1-n}\mathcal{F}_S(\{C_j\},B_r) - \rho^{1-n}\mathcal{F}_S(\{C_j\},B_\rho) \quad\\
     &=&      0
\end{eqnarray}
for almost all $r,\rho>0$.  Hence the sets $C_j$ differ only on a set of measure zero from cones
with vertices at the origin.  
\end{proof}

%%%%%%%%%%%%%%%%%%%%%%%%%%%%%%%%%%%%%%%%%%%%%%%%

\section{Tangent Plane to the Blow-Up Sphere}
\label{TP}

\begin{theorem}[See Proposition 9.6 of \cite{G}]
Suppose $\{C_j\}$ are blowup cones resulting from the limit process in Theorem \ref{thm:cones},
and let $x_0\in\partial C_0 \cap \partial C_1 \cap \partial C_2 \setminus \{0\}$.  For $t>0$, let
\begin{equation}
       \{C_{j,t}\} = \{x\in\R^n : \mathit{x_0 + t(x-x_0)\in\{C_j\} } \}. 
\end{equation}
Then there exists a sequence $\{t_i\}$ converging to zero such that $\{C_{j,i}\} := \{C_{j,t_i}\}$ converges to cones $\{Q_j\}$ which are a  V-minimizer of $\mathcal{F}_S$ in $\R^n$.  Moreover $\{Q_j\}$ are cylinders with axes through 0 and $x_0$.
\end{theorem}
\begin{figure}[!h]
	\centering
	\scalebox{.55}{\includegraphics{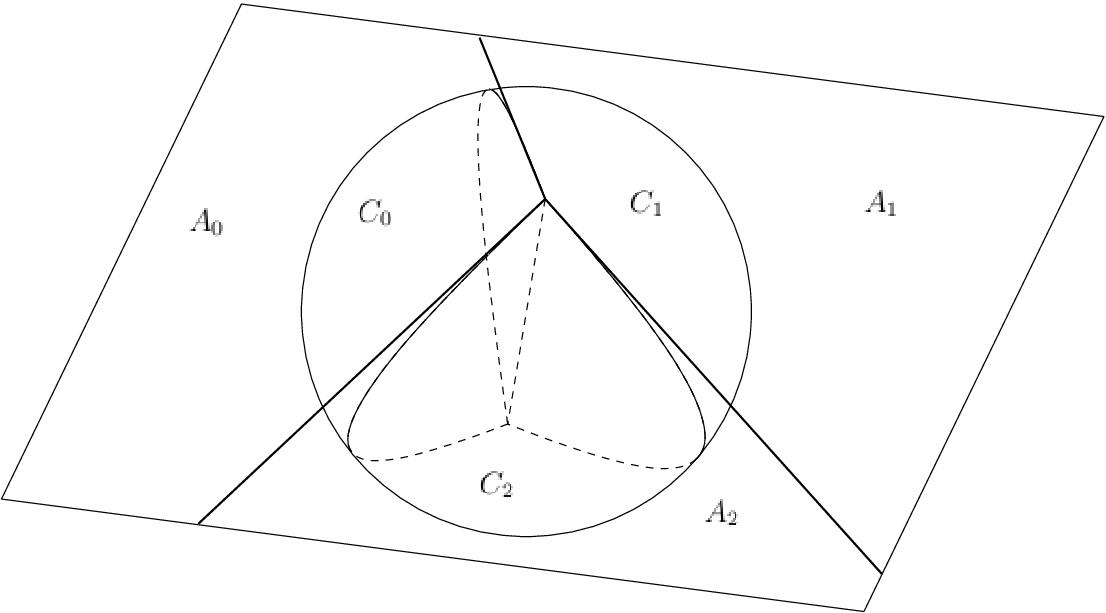}}
	\caption{Cones in the tangent plane to the blow up sphere.}
	\label{fig:ConesTangent}
\end{figure}
\begin{remark}[Existence of isolated triple points]  \label{EoITP}
It is not clear that we need to assume that a point such as $x_0$ exists in dimension 3.  Indeed, in dimension 3, we
conjecture that if there is a triple point in a minimal configuration of cones, then there will be a full line of these
triple points.  
%On the other hand, by taking $\alpha_j \equiv 1,$ and working in dimension 8 or higher, it is well-known that
%we can have an isolated triple point by having the so-called ``Simons Cone.''
%One can certainly conjecture that there are
%no isolated points in any dimension
%less than or equal to 7, but since the $\alpha_j$ do not have to equal each other, there is much more freedom here
%than in the simpler minimal surface problem.
\end{remark}
\begin{proof}
We may assume $x_0 = (0,0,...,0,a), a\ne 0$.  We have
\begin{equation}
     \chi_{C_{j,t}}(x) = \chi_{C_j}(x_0 + t(x-x_0))
\end{equation}
and so
\begin{equation}
      \mathcal{F}_S(\{C_{j,t}\},B(x_0,\rho) )= t^{1-n}\mathcal{F}_S(\{C_j\},B(x_0,\rho t) )
   = \rho^{n-1}\mathcal{F}_S(\{C_j\},B(x_0,1))
\end{equation}
The argument in the proof of Theorem~\ref{thm:cones} implies the existence of a sequence $\{t_i\}$ converging to 0 such that $\{C_{j,i}\}$ converges to cones $\{Q_j\},$ each with a vertex at $x_0$, and that V-minimize $\mathcal{F}_S$ over $\R^n$.

It remains to prove that $\{Q_j\}$ are cylinders with axes through 0 and $x_0$.  This is equivalent to the existence of sets $\{A_j\}\subseteq\R^{n-1}$ such that $\{Q_j\} = \{A_j\}\times\R$.  
Because the $\{C_j\}$ are all cones with vertex at $0,$ we have $\langle x, D\chi_{C_j} \rangle = 0$ and hence
\begin{equation}
     aD_n\chi_{C_j} = -\langle x-x_0, D\chi_{C_j}\rangle.
\end{equation}
Thus
\begin{equation}
     |D_n\chi_{C_j}| \leq \frac{|x-x_0|}{|x_0|}|D\chi_{C_j}|
\end{equation}
and then
\begin{eqnarray}
      \sum_{j = 0}^2 \alpha_j \int_{B(x_0,\rho)} |D_n\chi_{C_{j,t}}| 
    &=& t^{1-n} \sum_{j = 0}^2 \alpha_j \int_{B(x_0,\rho t)} |D_n\chi_{C_{j}}| \\
    &\leq& \frac{t^{2-n} \rho}{|x_0|} \sum_{j = 0}^2 \alpha_j \int_{B(x_0,\rho t)} |D\chi_{C_{j}}| \\
    &=& \frac{t^{2-n}\rho}{|x_0|}\mathcal{F}_S\left(\{C_j\},B(x_0,\rho t)\right) \\
    &\leq& C\frac{\rho^n t}{|x_0|} \, .
\end{eqnarray}
Thus
\begin{equation}
D_n\chi_{Q_j} = \lim_{i\rightarrow\infty} D_n\chi_{C_{j,t_i}} = 0, \qquad j=0,1,2.
\end{equation}

However, for almost all $s<t$, by Theorem~\ref{thm:giusti2.4},
\begin{equation}
           \int_{\mathcal{B}_R}|\chi_{Q_{j,s}} - \chi_{Q_{j,t}}| \, d\mathcal{H}^{n-1} 
    \leq \int_{\mathcal{B}_R\times(s,t)} |D_n\chi_{Q_j}| = 0
\end{equation}
where $\chi_{Q_{j,r}}(y) = \chi_{Q_j}(y,r)$.  This implies the existence of sets
$A_j\subseteq \R^{n-1}$ such that for almost all $r$ and $s$ we have
\begin{equation}
      \chi_{Q_j}(y,s) = \chi_{Q_j}(y,r) = \chi_{A_j}(y)
\end{equation}
for $j=0,1,2$ and almost all $y\in\R^{n-1}$.  Thus
\begin{equation}
          Q_j = A_j\times\R, \quad j = 0,1,2.
\end{equation}
\end{proof}

Since $\{Q_j\}$ are cones, then for each $t>0$, $(y,s)\in\R^{n-1}\times \R$
\begin{equation}
\chi_{A_j}(ty) = \chi_{Q_j}(ty,ts) = \chi_{Q_j}(y,s) = \chi_{A_j}(y), \quad j = 0,1,2,
\end{equation}
which implies $\{A_j\}$ are also cones.  We consider the case where $x_0 \in \partial B_1.$
Then we get a blow up limit in the tangent plane at that point.  We now turn to
the task of classifying the behavior in this tangent plane.

\begin{figure}[!h]
	\centering
	\scalebox{.7}{\includegraphics{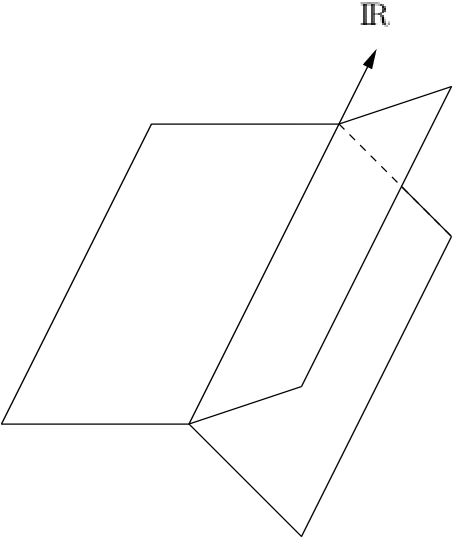}}
	\caption{Cylinders in the second blow up limit.}
	\label{fig:Cylinders}
\end{figure}
\begin{theorem}
Suppose $\{Q_j\} = \{A_j\} \times \R$ are V-permissible cylinders in $\R^n = \R^{n-1}\times\R$.
If $\{Q_j\}$ is a V-minimizer of $\mathcal{F}_S$ in $\R^n$ then $\{A_j\}$ is a V-minimizer of $\mathcal{F}_S$ in $\R^{n-1}.$
If we remove all of the volume constraints in the previous statements then the result still holds.
\end{theorem}
See Figure~\ref{fig:Cylinders}.\\
\begin{proof}
Without the volume constraints the proof only becomes simpler, so it suffices to
prove the statements where we include the volume restrictions.
Suppose $\{Q_j\}$ is a V-minimizer of $\mathcal{F}_S$ in $\R^n$.  If $\{A_j\}$ is not a V-minimizer of $\mathcal{F}_S$ in $\R^{n-1}$,
then there exists $\epsilon>0$, $R>0$, and sets $\{E_j\}$ coinciding with $\{A_j\}$ 
outside some compact set $H\subseteq \tilde{\mathcal{B}}_R$ such that 
\begin{equation}
     \mathcal{F}_S(\{E_j\},\tilde{\mathcal{B}}_R) \leq \mathcal{F}_S(\{A_j\},\tilde{\mathcal{B}}_R) - \epsilon.
\end{equation}
Let $T>0$ and set
\begin{equation}
     M_j = \begin{cases}
     E_j\times(-T,T) & \text{in } |x_n|<T \\
     Q_j & \text{outside } |x_n|<T
               \end{cases}
\end{equation}
for $j= 0,1,2$, giving $\{M_j\} = \{Q_j\}$ outside $H\times[-T,T]$.  Hence
\begin{equation}\label{eqn:Q<M}
            \mathcal{F}_S\left(\{Q_j\},\tilde{\mathcal{B}}_R\times[-T,T]\right) 
     \leq \mathcal{F}_S\left(\{M_j\},\tilde{\mathcal{B}}_R\times[-T,T]\right).
\end{equation}
However, we have
\begin{equation}
              \mathcal{F}_S\left(\{Q_j\},\tilde{\mathcal{B}}_R\times[-T,T]\right) 
      = 2T\mathcal{F}_S\left(\{A_j\},\tilde{\mathcal{B}}_R\right)
\end{equation}
and
\begin{eqnarray}
                 \mathcal{F}_S\left(\{M_j\},\tilde{\mathcal{B}}_R\times[-T,T]\right)
     &\leq& 2T\mathcal{F}_S\left(\{E_j\},\tilde{\mathcal{B}}_R\right) + 2\omega_{n-1}R^{n-1}\sum_{j=0}^2\alpha_j\\
     &\leq& 2T\mathcal{F}_S\left(\{A_j\},\tilde{\mathcal{B}}_R\right) - 2T\epsilon + 2\omega_{n-1}R^{n-1}\sum_{j=0}^2\alpha_j
\end{eqnarray}
This contradicts \eqref{eqn:Q<M} for sufficiently large $T$, say, $T>\frac{\omega_{n-1}}{\epsilon}R^{n-1}\sum_{j=0}^2\alpha_j$.
\end{proof}
Note that at this point we have proven everything in Theorem\refthm{mainresult}except the angle condition.

If we weaken our definition of minimality by abandoning the volume constraint again, then we are able to prove the converse.  We expect it is true with the volume constraint, but Figure~\ref{fig:thm84} illustrates the difficulty in generalizing the following proof.
Namely, the volume constraint could be satisfied globally, while individual slices did not preserve the induced
($n-1$)-dimensional volume constraints.
\begin{figure}[!h]
	\centering
	\scalebox{.6}{\includegraphics{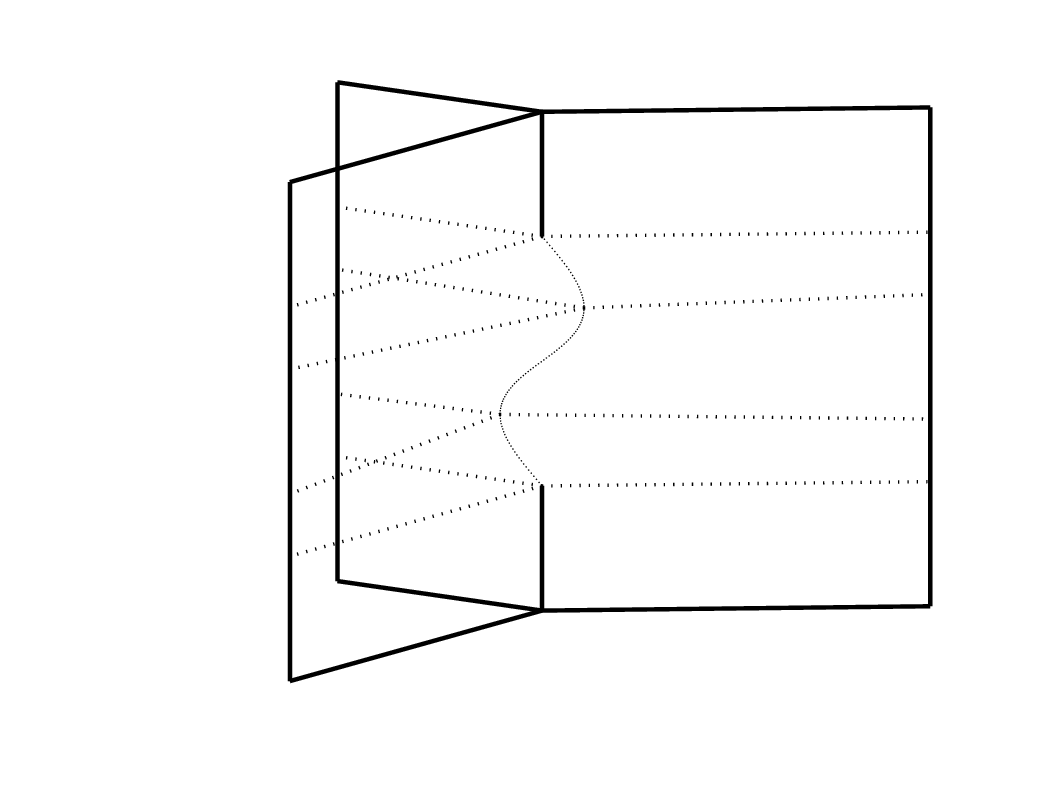}}
	\caption{A visualization of the difficulty in generalizing Theorem~\ref{thm:conemin}.  Perhaps each perpendicular slice is a minimizer in $\R^{n-1}.$}
	\label{fig:thm84}
\end{figure}
\begin{theorem}\label{thm:conemin}
Suppose $\{Q_j\} = \{A_j\}\times\R$ are permissible cylinders in $\R^n = \R^{n-1}\times\R$ without the volume constraint condition.
Then $\{A_j\}$ is minimal in $\R^{n-1}$ implies $\{Q_j\}$ is minimal in $\R^n$.
\end{theorem}
\begin{proof}
 Suppose $\{A_j\}$ is minimal in $\R^{n-1}$ and let $\{M_j\}$ be permissible
Caccioppoli sets in $\R^n$ coinciding with $\{Q_j\}$ outside some compact set $K$.
Recall that $\tilde{\mathcal{B}}_R$ denotes the ball in $\R^{n-1}$ centered at $0$ with radius $R,$
and choose $T>0$ such that 
\begin{equation}
%     K\subseteq \Lambda := \tilde{\mathcal{B}}_T\times(-T,T).
     K\subseteq \tilde{\mathcal{B}}_T\times(-T,T).
\end{equation}
Let $\{M_{j,t}\}\subseteq \R^{n-1}$ be defined by 
\begin{equation}
     \chi_{M_{j,t}}(y) = \chi_{M_j}(y,t), \quad j=0,1,2.
\end{equation}
Then Lemma~9.8 in \cite{G} gives
\begin{equation}
     \int_{A_j}|D\chi_{M_j}| \leq \int^T_{-T}\, dt \int_{\tilde{\mathcal{B}}_T}|D\chi_{M_{j,t}}|.
\end{equation}
Note $M_{j,t} = A_j$ outside compact sets $H_{j,t}\subseteq \tilde{\mathcal{B}}_T$ for
$j=0,1,2$, and $M_{j,t}$ are permissible.  Hence
\begin{equation}
       \mathcal{F}_S\left(\{A_j\},\tilde{\mathcal{B}}_T\right) \leq \mathcal{F}_S\left(\{M_{j,t}\},\tilde{\mathcal{B}}_T\right).
\end{equation}
Therefore
\begin{equation}
             \mathcal{F}_S\left(\{M_j\},\tilde{\mathcal{B}}_T\times(-T,T)\right) 
     \geq \sum_{j=0}^2 \alpha_j \int^T_{-T} \, dt \int_{\tilde{\mathcal{B}}_T}|D\chi_{A_j}|
          = \mathcal{F}_S\left(\{Q_j\},\tilde{\mathcal{B}}_T\times(-T,T)\right)
\end{equation}
which implies $\{Q_j\}$ is minimal.
\end{proof}

% % % % % % % % % % % % % % % % % % % % % % % % % % % % % % % % % % % % % % % % %
\section{Classification}
\label{classification}

We turn to classifying the possible minimal configurations in $\R^2$.  First we point out that using the
tools of mass-minimizing integral currents, Morgan \cite[Theorem 4.3]{Mo1998} showed that the triple
junction points are isolated in $\R^2$.  Under assumption of sufficient regularity Elcrat, Neel
and Siegel \cite{ENS} showed that the following Neumann angle condition holds
\begin{equation}\label{eq:ens}
\frac{\sin\gamma_{01}}{\sigma_{01}} = \frac{\sin\gamma_{02}}{\sigma_{02}} =
\frac{\sin\gamma_{12}}{\sigma_{12}},
\end{equation}
and their proof carries over to $\R^2$ directly.  Here $\gamma_{12}$ is the angle at the triple point measured within
$E_0$, $\gamma_{02}$ is the angle at the triple point measured within $E_1$, and $\gamma_{01}$ is the angle at
the triple point measured within $E_2.$
We are able to prove the following:
\begin{theorem}[Angle condition result]  \label{ACR}
   Let $\{ A_j\}$ be D-minimal or V-minimal cones in $\R^2$ with vertices at the origin.  Then each $A_j$ is formed of precisely
   one connected component, and the angle condition \eqref{eq:ens} is satisfied.
\label{thm:connectedcone}
\end{theorem}
\begin{corollary}[Volume constraints of blowups]  \label{VolConBlow}
No matter what volume constraints we impose on minimization problem, the angles at the triple points
for blowup limits are independent of everything except the constants which come from the surface tensions.
\end{corollary}
Note that the preceding two results wrap up the proof of Theorem\refthm{mainresult}\!\!.

We will start by dealing with the case without volume constraints, and 
we will proceed by contradiction, but first we will need some preliminary propositions.  Before the first proposition,
we record here a basic lemma which can be proven with no more than high school trigonometry:
\begin{lemma}[Basic trigonometry lemma]  \label{triglemma}
Given $\sigma_{01}, \; \sigma_{02}, \; \sigma_{12}$ satisfying the strict triangle inequality given in
Equation\refeqn{strtri}\!, there exists a unique triple of real numbers
$\Gamma_{01}, \; \Gamma_{02}, \; \Gamma_{12} \in (0,\pi)$ which satisfy both:
$$\Gamma_{01} + \Gamma_{02} + \Gamma_{12} = 2\pi$$
and Equation\refeqn{ens}\!\!.
\end{lemma}
\begin{proof}
We provide a sketch:  First,
because of the strict triangle inequality, there is a unique triangle (up to reflection and congruence, of course) with
sides with lengths given by $\sigma_{01}, \; \sigma_{02},$ and $\sigma_{12}.$  Now let $\theta_{ij}$
be the angle opposite $\sigma_{ij}.$  Then the law of sines gives us:
$$\frac{\sin \theta_{01}}{\sigma_{01}} = \frac{\sin \theta_{02}}{\sigma_{02}} = \frac{\sin \theta_{12}}{\sigma_{12}} \;.$$
Of course, the angles just given sum to $\pi$ and not $2\pi,$ but their supplementary angles sum to $2\pi$ and
have the same value when plugged into the sine function.  Define $\Gamma_{ij} := \pi - \theta_{ij}$ and
everything is satisfied.
\end{proof}

Now, most of the calculus that we need to do has to be done on a suitable triangle, so we start with
a definition of a ``good triangle'' and then give the calculus proposition which will be the main engine
in the rest of our proofs in this section.

\begin{definition}[Good Triangles]   \label{goodTs}
Given a blowup limit to our minimization problem, we define a good triangle to be a pair $(T,\tilde{P})$
consisting of a triangle $T$ (whose vertices we label as $P_0, \; P_1,$ and $P_2$), and a
point $\tilde{P}$ which is in the interior of the triangle such that the following hold:
\begin{enumerate}
    \item For $i,j \in \{0,1,2\}$ with $i \ne j$ we have that the angle between the vector from
             $\tilde{P}$ to $P_i$ and the vector from $\tilde{P}$ to $P_j$ is exactly $\Gamma_{ij}.$
    \item If $\{i,j,k\}$ is a permutation of $\{0,1,2\},$ then the open segment from $P_{i}$ to $P_{j}$
             has the $\text{k}^{\text{th}}$ fluid as data.
\end{enumerate}
In order to simplify the exposition, we can assume without loss of generality that the
ordering of the vertices is counter-clockwise with respect to the triangle $T.$  See Figure~\ref{fig:triangleT}.
\end{definition}
\begin{figure}[!h]
	\centering
	\scalebox{1.2}{\includegraphics[trim=0cm 1.15cm 0cm .3cm, clip]{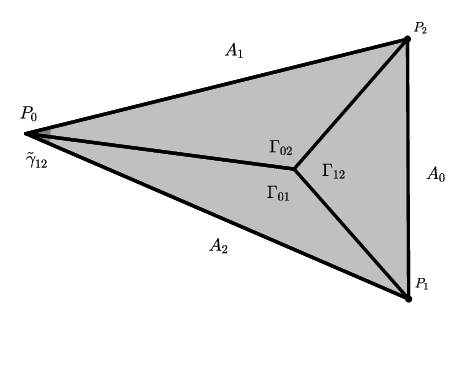}}
	\caption{The triangle $T$.}
	\label{fig:triangleT}
\end{figure}
\begin{definition}[Basic Cost Function]  \label{BCF}
Given any good triangle $(T,\tilde{P})$ with the vertices of $T$ labeled as $P_j := (a_j,b_j),$
and where for the sake of simplifying notation we let
$$\zeta_0 := \sigma_{12}, \ \ \ \zeta_1 := \sigma_{02}, \ \ \ \zeta_2 := \sigma_{01} \;,$$
we define the basic cost function by:
\begin{equation}
  C(x,y) := \sum_{j=0}^2 \zeta_j \sqrt{(x - a_j)^2 + (y - b_j)^2} \;.
  \label{eq:thecost}
\end{equation}
The cost function $C(x,y)$ is continuous on the closed bounded
triangle, $T,$ and so it must attain a minimum there.
\end{definition}

\begin{proposition}[Minimization on good triangles]  \label{MinGoodTs}
The unique $D-minimizer$ on a good triangle $(T,\tilde{P})$ with $\tilde{P} = (\tilde{x}, \tilde{y})$
is the configuration formed by letting $E_i$ be the
triangular region with $P_j, \; P_k,$ and $\tilde{P}$ as vertices, where we let $(i,j,k)$ run through the 
three permutations: $\{(0,1,2), (1,2,0), (2,0,1)\}.$  Furthermore, 
the basic cost function has the following properties:
	\begin{itemize}
                     \item[A)] The Hessian $D^2C$ is positive definite in the interior of $T.$
		\item[B)] $\nabla C(x,y)$ is zero if and only if $(x,y) = (\tilde{x},\tilde{y}).$
		\item[C)] The cost $C$ has a unique minimum at $(\tilde{x},\tilde{y}).$
		\item[D)] If we let $\tilde{v}_j$ denote the vector from $(\tilde{x},\tilde{y})$ to $(a_j, b_j),$ then the
		angles $$\gamma_{ij} := \arccos \frac{\tilde{v}_i \cdot \tilde{v}_j}{|\tilde{v}_i|\cdot|\tilde{v}_j|}$$
                     satisfy $$\gamma_{ij} \equiv \Gamma_{ij}$$ and therefore automatically
		obey the relation
		\begin{equation}
		\frac{\sin \gamma_{01}}{\sigma_{01}} = \frac{\sin \gamma_{12}}{\sigma_{12}} = \frac{\sin \gamma_{02}}{\sigma_{02}},
		\label{eq:ENSformula}
		\end{equation}
		which is derived by using the Calculus of Variations in \cite{ENS}.
	\end{itemize}
\end{proposition}

\begin{proof}
Our statements about the minimizer follow from our statements about the cost function, so we will skip
immediately to proving those facts.  In order to simplify our computations, we start by defining the following notation:
For $j = 0, \, 1,$ or $2,$
	we define:
	\begin{alignat*}{1}
	X_j &:= x - a_j \\
	Y_j &:= y - b_j \\
	\beta_j &:= (X_j^2 + Y_j^2)^{1/2} \\  %cube root of \beta_j from my notes
	W_j &:= \zeta_j / \beta_j \\
	Z_j &:= \zeta_j / \beta_j^3 \;.
	\end{alignat*}
	All sums are assumed to be sums from $j = 0$ to $2.$
	
	With our new notation, we easily compute:
	\begin{alignat*}{1}
	C(x,y) &= \sum \zeta_j \beta_j \\
	C_x(x,y) &= \sum W_j X_j  \\
	C_y(x,y) &= \sum W_j Y_j  \\
	C_{xx}(x,y) &= \sum Z_j Y_j^2 \\
	C_{xy}(x,y) &= \sum Z_j X_j Y_j \\
	C_{yy}(x,y) &= \sum Z_j X_j^2 \;.
	\end{alignat*}
          The trace of the Hessian $D^2 C$ is obviously strictly positive.
	The determinant of the Hessian $D^2 C$ is equal to
$$\left(\sum Z_j X_j^2 \right) \left(\sum Z_j Y_j^2 \right) - \left( \sum Z_j X_j Y_j \right)^2$$
and by using the Schwarz inequality on a set of three points with delta measures on each one
weighted by $Z_j$ we easily see that this determinant is nonnegative.  On the other hand,
by noting that equality in the Schwarz inequality only happens when one function is a multiple
of the other, and that $(X_0,X_1,X_2) = \alpha(Y_0,Y_1,Y_2)$ would mean that the slopes
of the vectors from $\tilde{P}$ to each vertex are the same, we can easily rule out equality,
and so we conclude that $D^2 C$ is positive definite, thereby proving A.
	
At this point we note that D follows immediately by definition of what a good triangle is, and
B implies C, now that we have our statement about the Hessian.
In fact, it will suffice to show that the gradient vanishes at $(\tilde{x},\tilde{y}),$ as the
uniqueness of the critical point of the cost function follows from positive definiteness of the Hessian.
Thus, the gradient condition that we now need to show is equivalent to showing that:
	\begin{equation}
	0 = \sum W_j X_j = \sum W_j Y_j \;
	\label{eq:gradcond}
	\end{equation}
holds when $(x,y) = (\tilde{x}, \tilde{y}).$

	We compute the sines of the angles, $\gamma_{ij},$ by giving them a zero z-component and then taking cross
	products (while carefully following the right-hand rule and recalling our convention about
           the counter-clockwise orientation of $(P_0,P_1,P_2)$):
	\begin{equation}
	\sin \gamma_{01} = \frac{(v_0 \times v_1)\cdot \hat{k}}{|v_0|\cdot|v_1|}
	= \frac{X_0 Y_1 - X_1 Y_0}{\beta_0 \beta_1} \;,
	\label{eq:sin01}
	\end{equation}
	\begin{equation}
	\sin \gamma_{12} = \frac{(v_1 \times v_2)\cdot \hat{k}}{|v_1|\cdot|v_2|}
	= \frac{X_1 Y_2 - X_2 Y_1}{\beta_1 \beta_2} \;,
	\label{eq:sin12}
	\end{equation}
	\begin{equation}
	\sin \gamma_{02} = \frac{(v_2 \times v_0)\cdot \hat{k}}{|v_2|\cdot|v_0|}
	= \frac{X_2 Y_0 - X_0 Y_2}{\beta_2 \beta_0} \;,
	\label{eq:sin02}
	\end{equation}
	where $\hat{k}$ is, as usual, the unit vector in the positive z direction.
	
	Observe that
	\begin{equation}
	\frac{\sin \gamma_{01}}{\sigma_{01}} = \frac{X_0 Y_1 - X_1 Y_0}{\zeta_2 \beta_0 \beta_1} \ \ \ \text{and} \ \ \ 
	\frac{\sin \gamma_{12}}{\sigma_{12}} = \frac{X_1 Y_2 - X_2 Y_1}{\zeta_0 \beta_1 \beta_2} \;.
	\label{eq:fracs}
	\end{equation}
Because we are assuming that we are at the point $(\tilde{x},\tilde{y}),$ we know that
	\begin{equation}
	\frac{X_0 Y_1 - X_1 Y_0}{\zeta_2 \beta_0 \beta_1} =
	\frac{X_1 Y_2 - X_2 Y_1}{\zeta_0 \beta_1 \beta_2} \;.
	\label{eq:fracsequal}
	\end{equation}
	By cross multiplication and some cancellation of the $\beta_1$ we see that we have
	\begin{equation}
           W_0 (X_0 Y_1 - X_1 Y_0) =
	\frac{\zeta_0 (X_0 Y_1 - X_1 Y_0)}{\beta_0} = 
	\frac{\zeta_2 (X_1 Y_2 - X_2 Y_1)}{\beta_2} = W_2 (X_1 Y_2 - X_2 Y_1) \;.
	\label{eq:cmquants}
	\end{equation}
           Now we note that
    \begin{alignat*}{1}
        W_2 (X_1 Y_2 - X_2 Y_1) &= W_0 (X_0 Y_1 - X_1 Y_0) \\
                                                 &= W_0 X_0 Y_1 - W_0 X_1 Y_0 + W_1 X_1 Y_1 - W_1 X_1 Y_1 \\
                                                 &= Y_1(W_0 X_0 + W_1 X_1) - X_1(W_0 Y_0 + W_1 Y_1) \\
                                                 &= Y_1(W_0 X_0 + W_1 X_1 + W_2 X_2) - W_2(X_2 Y_1) \\
                           & \ \ \ \ \ \ - X_1(W_0 Y_0 + W_1 Y_1 + W_2 Y_2) + W_2(X_1 Y_2)
    \end{alignat*}
and so we have
\begin{equation}
    0 = Y_1(W_0 X_0 + W_1 X_1 + W_2 X_2) - X_1(W_0 Y_0 + W_1 Y_1 + W_2 Y_2) \;.
\label{eq:conchere1}
\end{equation}
Arguing in the exact same way with each of the other combinations of angles, we see that:
\begin{equation}
    0 = Y_0(W_0 X_0 + W_1 X_1 + W_2 X_2) - X_0(W_0 Y_0 + W_1 Y_1 + W_2 Y_2)
\label{eq:conchere0}
\end{equation}
and
\begin{equation}
    0 = Y_2(W_0 X_0 + W_1 X_1 + W_2 X_2) - X_2(W_0 Y_0 + W_1 Y_1 + W_2 Y_2) \;.
\label{eq:conchere2}
\end{equation}
Here again, if both $W_0 X_0 + W_1 X_1 + W_2 X_2$ and $W_0 Y_0 + W_1 Y_1 + W_2 Y_2$
did not vanish, then we would come to a contradiction by having the slopes of $v_0, v_1,$ and
$v_2$ all equal.
Thus we have the nontrivial direction of B, and C follows.
\end{proof}

Now we turn to a task which is essentially Euclidean geometry which will allow us to produce a good triangle.
\begin{proposition}[Existence of Good Triangles] \label{ExistGoodTs}
Let $\{A_j\}$ be a permissible configuration of cones in $\R^2$ with vertices at the origin, and assume that
as we move through a counterclockwise rotation, we
have a sector which we will call $A_2$ which is a subset of $E_2,$
followed by a sector which we will call $A_0$ which is a subset of $E_0,$ 
followed by a sector which we will call $A_1$ which is a subset of $E_1.$
Furthermore, assume that the angle of the opening for $A_0$ is strictly less than the real number $\Gamma_{12}.$
Then letting $P_0$ be the origin, there exists a point $\tilde{P}$ within the infinite sector $A_0,$ such that
we can find a point $P_1$ on the ray between $A_0$ and $A_2$ and a point $P_2$ on the ray between $A_0$ and
$A_1$ such that the triangle formed with vertices given by the $P_j$ together with the point $\tilde{P}$
forms a good triangle.  See Figure~\ref{fig:BS}.
\end{proposition}
\begin{figure}[!h]
	\centering
	\scalebox{.65}{\includegraphics[trim=0cm 0cm 0cm 0cm, clip]{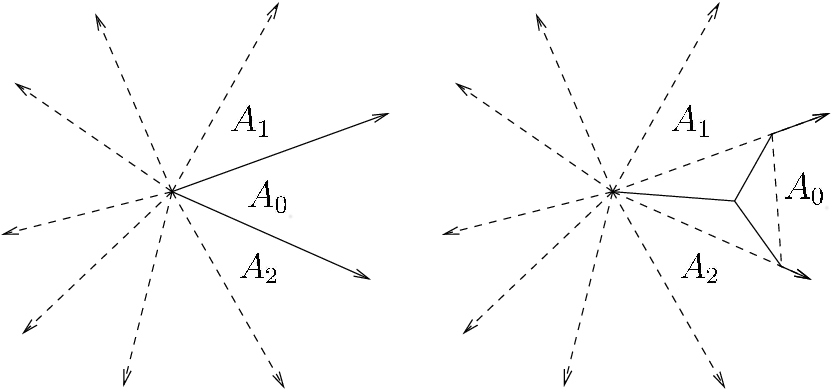}}
	\caption{The basic setting for proposition\refthm{ExistGoodTs}\!\!.}
	\label{fig:BS}
\end{figure}

\begin{remark}[Key Assumption]  \label{KA}
We note that after relabeling things and/or reflecting things we see that the only real assumption we make
is that the angle of the opening for $A_0$ is strictly smaller than $\Gamma_{12}.$
\end{remark}

\begin{proof}
We consider the set of points with distance one from the origin which intersects the solid sector $A_0,$
and we plan to make one of these points the point $\tilde{P}$ in our good triangle.
At each point on that set we extend three rays with the following two properties:
\begin{itemize}
    \item[1)] One of the rays passes through the origin.
    \item[2)] Going counter-clockwise from the rays passing through the origin,
                   the angles between the rays are $\Gamma_{01}$ followed by $\Gamma_{12}$ followed by
                   $\Gamma_{02}.$
\end{itemize}
Going counter-clockwise starting with the ray that passes through the origin, we will refer to these
rays as the ``zeroth ray,'' the ``first ray,'' and the ``second ray,'' respectively.
\begin{figure}[!h]
	\centering
	\scalebox{.6}{\includegraphics[trim=1.5cm 1.5cm 0cm 0cm, clip]{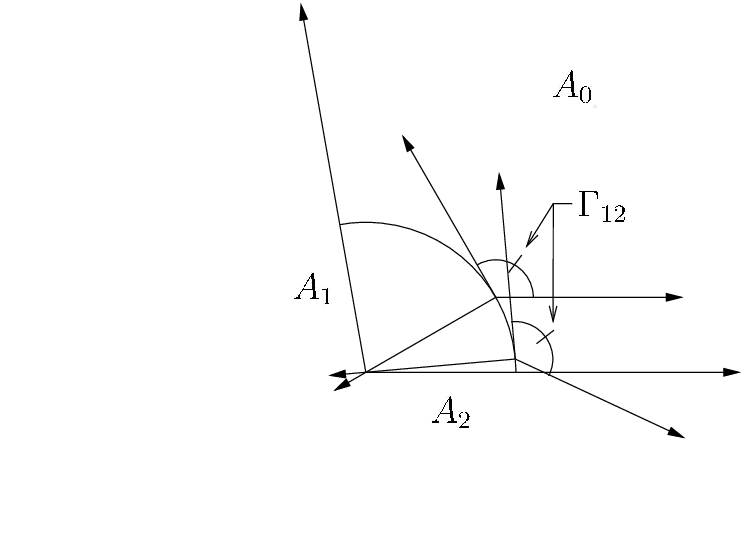}}
	\caption{The basic picture.}
	\label{fig:RF}
\end{figure}
In Figure~\ref{fig:RF} we have placed these three rays for two of the points with distance one from the origin
and we have labeled the angle which has measure $\Gamma_{12}.$  We can choose
our coordinate system so that the border between $A_0$ and $A_2$ is the positive x-axis, and then owing
to the fact that all of the $\Gamma_{ij} < \pi,$ we see that with $\tilde{P}$ having $\theta > 0$ but sufficiently
small, we must have an intersection of the first ray with $\partial A_0 \cap \partial A_2.$  In Figure~\ref{fig:RF}
we have chosen one of our potential $\tilde{P}$'s to have $\theta$ equal to five degrees.  Now if the second ray
has a nonempty intersection with $\partial A_0 \cap \partial A_1,$ then we are done by letting $P_1$
and $P_2$ be the two points of intersection that we have found already. On the other hand, it is not necessarily the
case that the second ray will intersect $\partial A_1$ if $\theta$ is sufficiently small.  Assuming that there is no
intersection we consider what happens as we increase $\theta$ while recalling that the main hypothesis guarantees
that $\Gamma_{12}$ is larger than the angle between the rays on either side of $A_0.$  In particular, this hypothesis
guarantees that the second ray will be parallel to $\partial A_1 \cap \partial A_0$ at a value of $\theta$ which we
can call $\theta_1$ which is strictly less than the value of $\theta$ which we call $\theta_2$ where the first ray
is parallel to $\partial A_2 \cap \partial A_0.$  Then by taking $\theta$ strictly between $\theta_1$ and $\theta_2$
we are guaranteed a frame of three vectors with all of the desired intersections.  (In Figure~\ref{fig:RF} we
have chosen $\theta_2$ as another value of $\theta$ where we plotted the three relevant rays.  Decreasing $\theta$
from that value very slightly gives us what we need.)
\end{proof}

The next proposition we need shows that at blow up limits we have a distinct sector for each fluid
and not multiple sectors for any of the fluids.
\begin{proposition}[One Sector Per Fluid at Blowups]  \label{OSPFATPB}
     Let $\{ A_j\}$ be D-minimal cones in $\R^2$ with vertices at the origin.  Then each $A_j$ is formed of
     precisely one connected component.
\end{proposition}

\begin{proof}
The first observation we need is that if we don't have all three fluids in any three consecutive sectors, then the
triangle inequality guarantees an improvement by ``filling in'' near the triple point.  See Figure~\ref{fig:All3Diff}
where we have only $A_0$ and $A_1$ in three consecutive sectors on the left hand side, and where we have
an immediate improvement on the right hand side.
\begin{figure}[!h]
	\centering
	\scalebox{1}{\includegraphics[trim=0cm 0cm 0cm 0cm, clip]{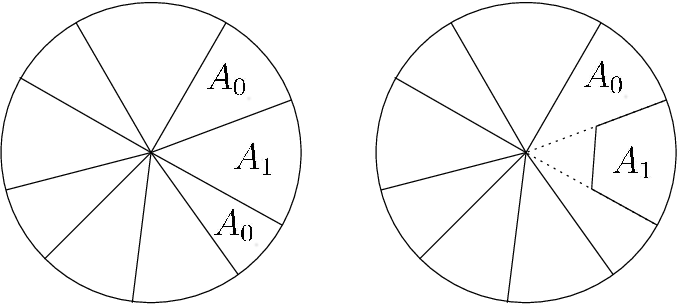}}
	\caption{An improvement when three consecutive sectors have only two fluids.}
	\label{fig:All3Diff}
\end{figure}
Thus, it follows that if we have more than three sectors, then we must have at least six sectors.

Now by renaming and/or relabeling we can assume without loss of generality that we have the situation
depicted on the left hand side of Figure~\ref{fig:Improve}.   Furthermore, using the fact that we have at least
six sectors now, we can assume that the angle of the sector for $A_0$ on the left hand side of the figure
is less than or equal to $\Gamma_{12}/2$ which is strictly less than $\Gamma_{12}.$  Now of course
we can apply Proposition\refthm{ExistGoodTs}to get the existence of a good triangle, followed by
Proposition\refthm{MinGoodTs}to come to a contradiction.
\begin{figure}[!h]
	\centering
	\scalebox{1}{\includegraphics[trim=0cm 0cm 0cm 0cm, clip]{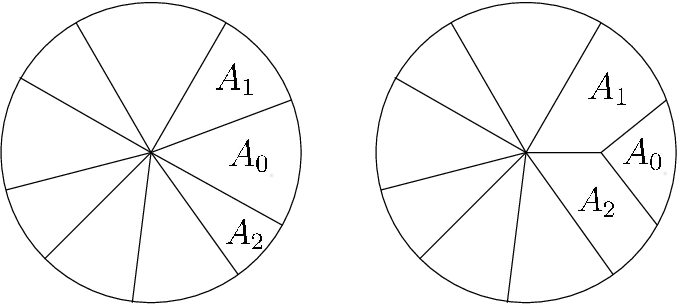}}
	\caption{An improvement when three consecutive sectors have only two fluids.}
	\label{fig:Improve}
\end{figure}
\end{proof}

Now we can turn to the proof of Theorem~\ref{thm:connectedcone}:

\noindent
\begin{proof}
In fact, at this point, the D-minimal situation is essentially complete.  The final observation needed is that if
the angles are not exactly what they are supposed to be, then one of the angles is smaller than the corresponding
$\Gamma_{ij}$ and then (after renaming the indices if necessary) we can invoke
Proposition\refthm{ExistGoodTs}followed by Proposition\refthm{MinGoodTs}to achieve the desired result. Thus,
we turn immediately to the V-minimal case.

The key observation in the V-minimal case is that we can actually improve Almgren's Volume Adjustment Lemma
by removing any lack of uniformity when our configuration consists solely of cones.  Indeed, we suppose toward
a contradiction that we have a V-minimal configuration of cones which does not satisfy the angle condition.  In this
case, it follows from the D-minimal proof that we can lower the energy by some amount within $B_1$ if we temporarily
ignore the volume constraint.  On the other hand, by considering our sectors on a large enough disk, we can restore the
volume constraint by adding or subtracting rectangles along the boundaries of the sectors at a cost which is bounded
by twice the width of the rectangle times the largest $\sigma_{ij}.$  See Figure~\ref{fig:RectPic}.  
Of course, since we can choose our disk to be as
large as we like, our rectangles can have arbitrarily small width, and therefore we can fix the volume constraint with
a loss to our energy which is as small as we like.  The arbitrarily small width that we can have for these
rectangles also guarantees that even if one of our sectors is very thin, by shrinking the width of the rectangle if 
necessary, we do not have to worry about having an intersection with more than one of the rays bounding our sector.
Thus, the original configuration could not possibly have been the
V-minimizer on our large disk and that gives us the desired contradiction.
\begin{figure}[!h]
	\centering
	\scalebox{.75}{\includegraphics[trim=8cm 5cm 0cm .75cm, clip]{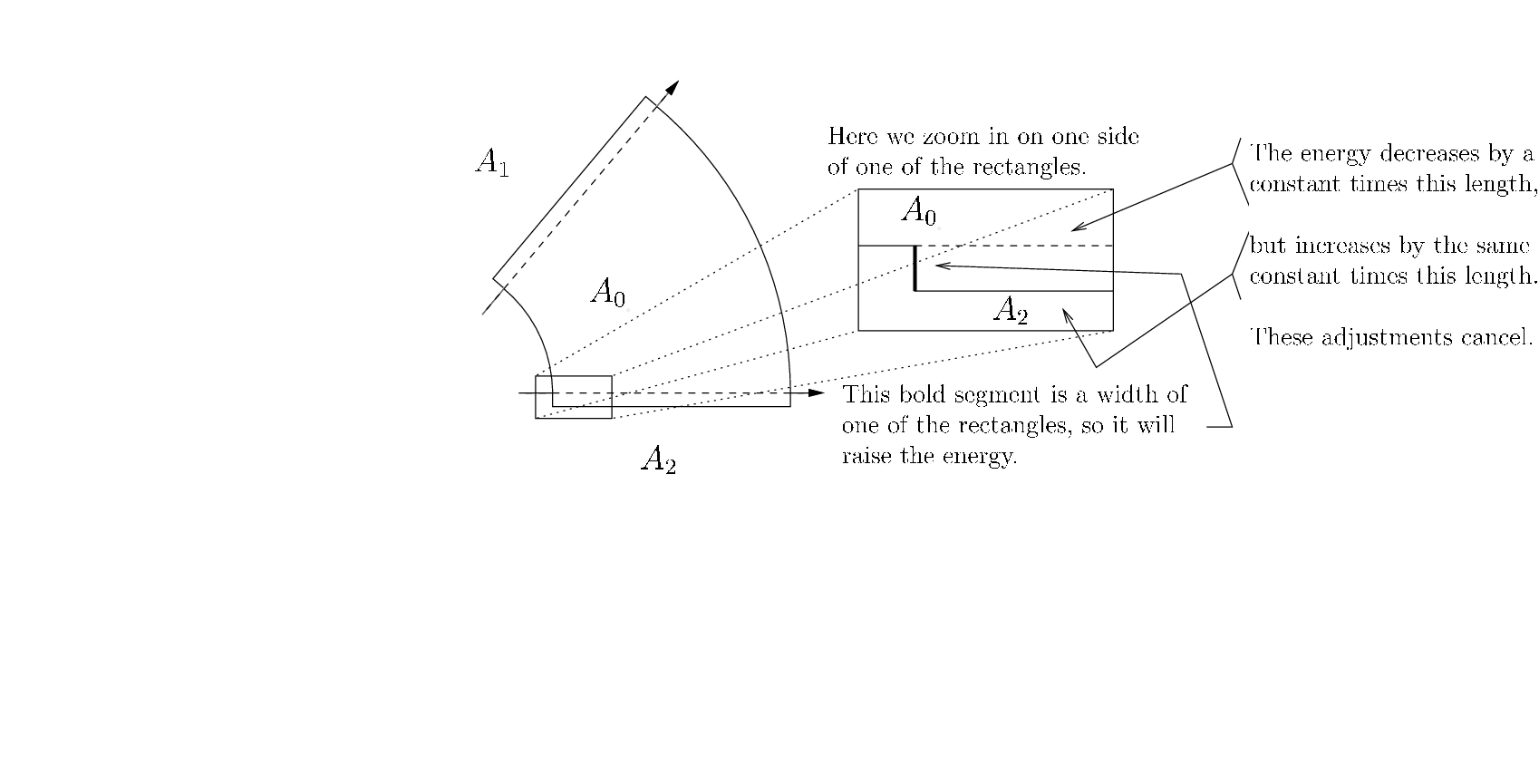}}
	\caption{An example of using rectangles to adjust the volumes.}
	\label{fig:RectPic}
\end{figure}
\end{proof}

\begin{remark}[Rectangles are not optimal, but very convenient]  \label{Ranobvc}
The competing variations built by rectangles could be immediately improved by using
smoother connections to the old boundary, however, we prefer the explicit construction
presented in the proof.
\end{remark}

We note that Futer, Gnepp, McMath, Munson, Ng, Pahk, and Yoder \cite{FGMMNPY} studied
planar cones that are minimizing, which is similar to some of the results above, however they
proved their results using a calibration argument as in \cite{LM1994}.  Lawlor and Morgan
give a criteria for a configuration of immiscible fluids to be energy minimizing
(see \cite[equation 1, section 1.2]{LM1994}) in the case where the interfaces are pieces of
planes, and presumably, this equation is equivalent to the Neumann angle condition in dimension
2 or 3, although they make no direct claims of this fact.
Using Morgan's result that there are only finitely many triple points in the tangent
plane \cite{Mo1998}, we may use our results to conclude that there are only finitely
many triple points on the blow up sphere $\partial B_1$.  Classifying this finite number
of free boundary points remains an open problem.

%We would also like to comment that we have not succeeded in showing the uniqueness
%of our blow up cones.  The introduction of Almgren's Big Regularity Paper \cite{Am2000}
%discusses this difficulty, and some examples of the slowly rotating configurations are in
%Leonardi \cite{L2002}.  In one related setting there has been success in showing that
%the tangent cone is unique: See White \cite{W1983}.

\section{Concluding comments}
\label{Acknow}

It is with great sadness that we must report that our collaborator Alan Elcrat passed away
suddenly on December 20th, 2013.  He was an energetic and hard-working mathematician,
and a good friend and mentor.  It is without doubt that the current work would not have been
completed without him, and that future works will be more difficult without his insight.

Finally, to close with some cheer, we wish to thank Luis Silvestre and especially Frank Morgan for
useful conversations.  Silvestre helped us with certain aspects of the coarea formula, and Morgan
assisted us greatly in understanding Allard's work.   We also wish to thank the referees for their
expertise with geometric measure theory and for their very constructive criticisms of earlier drafts
of this work.  Finally, the third author was a postdoc at Kansas State University when this project
began, and he was also partially supported by an REP grant from Texas State University in 2012
for work on this project.
%Insofar as it has provided us with great inspiration, we would like to thank the Creedmore Dumps
%for being named the Creedmore Dumps.

%%%%%%%%%%%%%%%%%%%%%%%%%%%%%%%%%%%%%%%%%%%%%

\noindent
Kansas State University, Department of Mathematics, Manhattan, Kansas, blanki@math.ksu.edu \\ \ \\
Texas State University, Department of Mathematics, San Marcos, Texas, rt30@txstate.edu

\end{document}